\documentclass[11pt,a4paper]{article}
\usepackage[utf8]{inputenc}
\usepackage{amsmath}
\usepackage{bm}
\usepackage{amsfonts}
\usepackage{titlesec}
\usepackage[capposition=top]{floatrow}
\usepackage[margin=1in]{geometry}
\usepackage{changepage}
\usepackage{amssymb}
\usepackage{relsize}
\usepackage{xcolor}
\usepackage{marginnote}
\usepackage{expl3}
\usepackage{mathtools}
\usepackage{amsthm}
\usepackage{bbm}
\usepackage{shadethm}
\usepackage{graphicx}
\usepackage{setspace}
\usepackage{braket}
\usepackage{epstopdf}
\usepackage{titlesec}
\usepackage{pdfpages} 
\let\ppprod\prod
\let\sssum\sum
\let\inttt\int
\let\bigcuppp\bigcup
\let\bigsqcuppp\bigsqcup
\usepackage{OldStandard} 
\renewcommand{\epsilon}{\varepsilon}

\renewcommand{\bigcup}{\mathlarger{\bigcuppp}}
\renewcommand{\bigsqcup}{\mathlarger{\bigsqcuppp}}
\renewcommand{\int}{\mathlarger{\inttt}} 
\renewcommand{\prod}{\mathlarger{\ppprod}} 
\renewcommand{\sum}{\mathlarger{\sssum}} 

\let\originalleft\left 
\let\originalright\right
\renewcommand{\left}{\mathopen{}\mathclose\bgroup\originalleft}
\renewcommand{\right}{\aftergroup\egroup\originalright}

\newcommand{\dd}{\mathrm{d}}

\newcommand{\lb}{\left[}
\newcommand{\rb}{\right]}
\newcommand{\lp}{\left(}
\newcommand{\rp}{\right)}
\newcommand{\lbr}{\left\lbrace}
\newcommand{\rbr}{\right\rbrace}

\newcommand{\hsp}{\hspace{0.1cm}}

\newcommand{\R}{\mathbb{R}}
\newcommand{\Q}{\mathbb{Q}}
\newcommand{\Z}{\mathbb{Z}}
\newcommand{\C}{\mathbb{C}}

\newcommand{\ringint}{\mathcal{O}}
\newcommand{\fracp}{\mathfrak{p}}

\newcommand{\norm}{\mathbb{N}}
\newcommand{\K}{\mathbb{K}}

\newcommand{\family}{\mathcal{F}}

\newcommand{\im}{\mathrm{Im}}
\newcommand{\re}{\mathrm{Re}}
\numberwithin{equation}{section}
\titleformat
{\section} 
[display] 
{\scshape\large} 
{} 
{0.5ex} 
{\thesection. \hsp \hsp \centering} 
[] 
\titleformat
{\subsection} 
[display] 
{\scshape} 
{} 
{0.5ex} 
{\thesubsection. \hsp} 
[] 
\titleformat
{\subsubsection} 
[display] 
{\scshape} 
{} 
{0.5ex} 
{\thesubsubsection. \hsp} 
[] 
\newtheoremstyle{sltheorem}
{1.2em}                
{1.2em}                
{\slshape}        
{}                
{\scshape}       
{.}               
{.5em}               
{}                
\theoremstyle{sltheorem}

\newtheorem{theorem}{Theorem}

\newtheorem{prop}[theorem]{Proposition}
\newtheorem{lemma}[theorem]{Lemma}

\numberwithin{theorem}{section}

\definecolor{shadethmcolor}{cmyk}{0,0,0,0}
\definecolor{shaderulecolor}{cmyk}{0,0,0,1} 

\allowdisplaybreaks

\makeatletter
\renewcommand\tableofcontents{%
    \@starttoc{toc}%
}
\makeatother
\begin{document}
\thispagestyle{empty}
\begin{center}
\noindent
\textbf{THE $1$-LEVEL DENSITY FOR ZEROS OF HECKE \\ L-FUNCTIONS OF IMAGINARY QUADRATIC \\ NUMBER FIELDS OF CLASS NUMBER $1$}\\ \vspace{0.3cm}
\textsc{Kristian Holm}\footnote{\textit{email:} {\tt holm@math.uni-kiel.de}} \quad \quad \today
\end{center} 
\begin{adjustwidth}{2em}{2em}
\begin{small}
\textbf{Abstract.} Let $\mathbb{K} = \Q ( \sqrt{-d} )$ be an imaginary quadratic number field of class number $1$ and $\ringint_\K$ its ring of integers. We study a family of Hecke $L$-functions associated to angular characters on the non-zero ideals of $\ringint_\K$. Using the powerful Ratios Conjecture (RC) due to Conrey, Farmer, and Zirnbauer, we compute a conditional asymptotic for the average $1$-level density of the zeros of this family, including terms of lower order than the main term in the Katz--Sarnak Density Conjecture coming from random matrix theory. We also prove an unconditional result about the $1$-level density, which agrees with the RC prediction when our test functions have Fourier transforms with support in $(-1,1)$. \\
\end{small}
\end{adjustwidth}
\begin{footnotesize}
\end{footnotesize}
\section{Introduction}
Many problems in modern number theory await progress due to the difficulty of obtaining exact information about zeros of $L$-functions. Perhaps not unrelated to this difficulty, the study of the large scale statistics of such zeros has also become a topic of much interest, the underlying philosophy being that a collection of objects is often more regular and well-behaved than the objects themselves. This line of research began with the work of Montgomery in the 1970's who famously conjectured \cite{montgomery} that the \textit{pair correlation} of normalized zeros of the Riemann zeta function, quantifying the "probability" of $\zeta$ having two zeros within some prescribed distance of each other, is the same as the pair correlation for eigenvalues of random Hermitian matrices. Today, there is a large body of conjectures describing the links between $L$-functions and random matrices, or formulating properties that should hold for $L$-functions by analogy with random matrices. In this paper, we will focus on two of these: The Katz--Sarnak Density Conjecture and the $L$-functions Ratios Conjecture. \par 
The Katz--Sarnak conjecture is a statement about another statistic of the zeros of (a family of) $L$-functions, namely the \textit{$1$-level density}, which was first studied by Özlük and Snyder in \cite{osnyder}. For our purposes, it can be defined as follows: If $\mathcal{F} = \lbr L_k : k\geqslant 1 \rbr$ is a family of $L$-functions indexed by some parameter $k$, and $\mathcal{F}(K) = \lbr L_k : 1 \leqslant k \leqslant K \rbr$, let
\begin{align*}
Z_k := \lbr z : 0 \leqslant \re(z) \leqslant 1, \,  L_k(z) = 0 \rbr
\end{align*}
be the set of zeros of $L_k$ in the critical strip. If $\rho \in Z_k$, let $\gamma (\rho) = -i \lp \rho - 1/2 \rp$. Thus, under the Riemann Hypothesis for the family $\mathcal{F}(K)$, $\gamma (\rho)$ is the imaginary part of the zero $\rho$ of $L_k$. Furthermore, let $f : \R \rightarrow \R$ be an even Schwarz function with the property that its Fourier transform $\widehat{f}$ has compact support. Then the \textit{$1$-level density} of the zeros of the family $\mathcal{F}(K)$ is the number
\begin{align*}
D \lp \mathcal{F}(K); f \rp := \frac{1}{K} \sum_{k = 1}^K \, \sum_{\rho \in Z_k} f \lp \frac{\gamma ( \rho) \log K}{\pi} \rp.
\end{align*}
Here, the scaling of $\gamma ( \rho)$ ensures that the average spacing between the zeros is approximately $1$. \par 
Under the Grand Riemann Hypothesis, the $1$-level density measures the average density of the normalized zeros of the family $\mathcal{F}(K)$ \textit{in a weak sense}. That is, for zeros on the critical line $\re (z) = 1/2$, and especially such zeros close to the real line, $D( \mathcal{F}(K), \cdot )$ is a functional that sees their (scaled) distribution through the lens of a suitable test function. When taking more and more $L$-functions of the family and their zeros into account, one thus obtains a sequence of distributions. For the purposes of studying this sequence, and in particular its limit, such weak characterizations in fact give a complete picture of the distribution of the zeros of the family. (At least, this holds if one knows the weak characterizations of the distribution for a sufficiently large class of test functions, thanks to classic results in probability theory such as the Portmanteau theorem.) \par 
Regarding the limiting distribution of the zeros of a family of $L$-functions, Katz and Sarnak conjectured \cite{ks1, ks2} that when $K$ tends to infinity, the functional $D ( \mathcal{F}(K), \cdot )$ converges weakly to some integral kernel that arises in the dimensional limit of the $1$-level density of eigenangles of random unitary matrices, chosen uniformly at random with respect to Haar measure either from the full unitary group $\mathrm{U}(N)$ or from one of the subgroups $\mathrm{USp}(N)$ (when $N$ is even), $\mathrm{O}(N)$, or $\mathrm{SO}(N)$. Specifically, if the family $\mathcal{F}$ has so-called \textit{unitary symplectic symmetry type}, the \textit{Katz--Sarnak Density Conjecture} states that
\begin{align}\label{katzsarnak}
\lim_{K \rightarrow \infty} D \lp \mathcal{F}(K); f \rp = \int_{\R} f(x) \lp 1 - \frac{\sin(2 \pi x)}{2 \pi x} \rp \hsp \dd x
\end{align}
for any even Schwarz function $f$ whose Fourier transform has compact support. We emphasize that if the family $\mathcal{F}(K)$ has a different symmetry type, the integral kernel which is conjectured to appear in the limit has a different form. \par 
Next, the very powerful $L$-functions Ratios Conjecture due to Conrey, Farmer, and Zirnbauer \cite{confarzir} asserts that averages of quotients of $L$-functions evaluated at certain parameters satisfy asymptotics that parallel those of quotients of characteristic polynomials of matrices. (See \textsc{Section 3} for a more detailed statement.) Many authors have used the Ratios Conjecture to study statistical aspects of the zeros of $L$-functions, or various other aspects of such functions. For example, Conrey and Snaith \cite{consnaith} studied the pair correlation of the zeros of the Riemann zeta function. In the same paper, they also studied the $1$-level density for zeros of quadratic Dirichlet $L$-functions. Later, the Ratios Conjecture was used to study zeros of $L$-functions of a more general class of characters, namely \textit{Hecke characters} of a number field. The Hecke $L$-functions considered in \cite{waxman} are those associated to angular characters of the Gaussian integers, and Waxman here used the Ratios Conjecture to compute the $1$-level density and identify lower-order terms (compared with the Katz--Sarnak heuristic) in this asymptotic. The goal of this paper is to follow \cite{waxman} and do such a study for a general imaginary quadratic number field of class number $1$. Thus, we consider a family $\mathcal{F}(K)$ of $L$-functions associated to angular Hecke characters of such fields, which we will describe now. \par 
By the Baker--Heegner--Stark Theorem, a complete list of imaginary quadratic number fields with class number $1$ is given by $\K := \K_d := \Q (\sqrt{-d})$, where $d$ is one of the \textit{Heegner numbers},
\begin{align*}
d = 1, 2, 3, 7, 11, 19, 43, 67, 163.
\end{align*}
Since the case $d = 1$ has already been treated in \cite{waxman}, we will let $d$ denote any of the eight remaining numbers on the list above. (This restriction will also make certain computations simpler, as the arguments involve several functions defined conditionally on the value of $d$.) Moreover, we will also let $N\geqslant 1$ denote any fixed positive multiple of $\left| \ringint_\K^\times \right|$, where $\left| \ringint_\K^\times \right| < \infty$ is the order of the group of units in the ring $\ringint_\K$. Our family of $L$-functions is then given by $\mathcal{F} = \lbr L_k(s) : k \geqslant 1 \rbr$, where
\begin{align*}
L_k(s) := \sum_{I \subset \ringint_\K \atop I \neq 0} \frac{\psi_k(I)}{\norm (I)^s}, \quad \quad \psi_k(\langle \alpha \rangle) = \lp \frac{\alpha}{\overline{\alpha}} \rp^{Nk},
\end{align*}
when $\re(s) > 1$. \par 
We note that such $L$-functions have been studied for arithmetic purposes on several occasions in the past. To give a few examples, Harman and Lewis considered the functions $L(s, \Xi_k)$, $k \geqslant 1$, with the Hecke character $\Xi_k$ given by $\Xi_k (\alpha ) = \lp \alpha / \, \overline{\alpha} \rp^{2k}$ for $\alpha \in \Z \lb i \rb$, and proved \cite[Thm. 1]{harmanlewis} the existence of infinitely many rational primes $p$ that have a Gaussian prime factor with a small argument (depending on the size of $p$). Later, in \cite{rudwax} Rudnick and Waxman considered the same family of $L$-functions and counted Gaussian primes in more general sectors of the complex plane, in a sense quantifying Hecke's classical theorem about the equidistribution of the angles of Gaussian primes on the circle (\cite{heckeangles1}, \cite{heckeangles2}). In particular, the authors in \cite{rudwax} studied the variance of such smooth counts of Gaussian primes and conjectured an asymptotic (\cite[Conjecture 1.2]{rudwax}) for this statistic based on a random matrix model and an analogue with similar counts over function fields. The asymptotic behaviour of this variance was, in fact, investigated quite recently from a different point of view in \cite{waxmanetalnew}, with a particular point of interest being the nature of the lower order terms in the asymptotic. An important aspect of the work in \cite{waxmanetalnew} relied on a study of the Hecke $L$-functions $L(s, \Xi_k)$ discussed above in combination with the Ratios Conjecture.\par 
We now state our main results. In the formulations of these (and throughout the paper), $D$ denotes the discriminant of our number field $\K$, and $\chi (n) = (-d/n)$ denotes the Dirichlet character coming from the Kronecker symbol (see \textsc{Section} 2.3). Moreover, $\gamma$ denotes the Euler--Mascheroni constant.
\begin{theorem}\label{THMunconditional}
Suppose that $f : \R \rightarrow \R$ is an even Schwarz function with $\mathrm{supp} \, \widehat{f} \subset \lp -1, 1 \rp$. Then we have
\begin{align*}
D \lp \mathcal{F}(K); f \rp = \int_{\R} f(x) \lp 1 - \frac{\sin(2 \pi x)}{2 \pi x} \rp \hsp \dd x + \frac{\ell_0 \hat{f}(0)}{\log K} + O \lp \frac{1}{(\log K)^{2}} \rp,
\end{align*}
where
\begin{align*}
\ell_0 &= - \int_1^\infty t^{-2} \lp -t + \sum_{n \leqslant t} \Lambda (n) \rp \hsp \dd t - \frac{L'(1, \chi)}{L(1, \chi)} - 2 \sum_{p \geqslant 3 \atop (-d/p) = -1} \frac{\log p}{p^2 - 1} \\ &\quad \quad + \log \sqrt{|D|} - \log 2 \pi + \log N - 2 - \frac{\sqrt{d} \log d}{d-1}  
- \frac{2\log 2}{3} \cdot \mathbbm{1}(d \neq 2,7),
\end{align*}
and $\Lambda (n)$ denotes the von Mangoldt function.
\end{theorem}
By assuming the Grand Riemann Hypothesis (GRH) and the Ratios Conjecture (\cite{confarzir}), we also prove the following result.
\begin{theorem}\label{OneLevelDensityThm-RC}
Suppose that $f : \R \rightarrow \R$ is an even Schwarz function whose Fourier transform has compact support. Assume the GRH and the Ratios Conjecture. Then
\begin{align*}
\begin{split}
D \lp \mathcal{F}(K); f \rp = \int_{\R} f(x) \lp 1 - \frac{\sin(2 \pi x)}{2 \pi x} \rp \hsp \dd x + \frac{\ell_0}{\log K} \lp \hat{f}(0) - \hat{f}(1) \rp + O \lp \frac{1}{(\log K)^{2}} \rp,
\end{split}
\end{align*}
where $\ell_0$ is as in the statement of Theorem \ref{THMunconditional}.
\end{theorem}
\textsc{Remarks.} \par 
1) Although we give an explicit value for $L'(1, \chi)/L(1, \chi)$ in Lemma \ref{kroneckerlimitformulastuff}, this expression is rather intricate, and we therefore decided to keep the notation $L'(1, \chi) / L(1, \chi)$ in the statements of the theorems. \par 
2) We note that both Theorem \ref{THMunconditional} and Theorem \ref{OneLevelDensityThm-RC} verify the Katz--Sarnak Density Conjecture, but to different extents: While Theorem \ref{THMunconditional} requires $f$ to have a Fourier transform with very small support, Theorem \ref{OneLevelDensityThm-RC} holds without any such assumption.\par 
3) The appearance of $\left| \ringint_\K^\times \right|$ in the exponents of the characters $\psi_k$ is very natural, since $\psi_k$ must satisfy a condition related to the units in order to define a Hecke character on the ideals of $\ringint_\K$ (cf. \textsc{Section} 2.1). Compared to the setup in \cite{waxman}, we are considering a more general family of characters since we allow $N$ to be any multiple of the order of the unit group. The reason why we are able to handle this more general case is that we formulate and prove a generalization (Theorem \ref{anglesandlengthsyeeeah}) of the result \cite[Lemma 2.1]{rudwax} relating the arguments and norms of certain elements of $\ringint_\K$. \\ \par
Our approach is based on \cite{waxman}. We first prove Theorem \ref{OneLevelDensityThm-RC}, and this is accomplished in \textsc{Section 3}, where we also describe the Ratios Conjecture in detail for our family $\mathcal{F}$, and in \textsc{Section 4}. In \textsc{Section 5}, we use the explicit formula for our family $\mathcal{F}$ to give an unconditional asymptotic for the $1$-level density. \textsc{Section 6} is then a comparison between this and the conditional asymptotic, which leads to a proof of Theorem \ref{THMunconditional}. As mentioned, this comparison is facilitated by Theorem \ref{anglesandlengthsyeeeah}, which we also state and prove in \textsc{Section 6}. \\\par 
\textsc{Acknowledgements.} We are most grateful to Anders Södergren for suggesting the problem considered in this article, and for many helpful discussions. Moreover, we are indebted to Victor Ahlquist for pointing out the symmetry between the coefficients of $\widehat{f}(0)$ and $\widehat{f}(1)$ in the statement of Theorem \ref{OneLevelDensityThm-RC}. Finally, we would like to thank Ezra Waxman, Daniel Fiorilli, Julia Brandes, and Michael Björklund for valuable discussions and comments in relation to this project.
\section{Preliminaries}
We will now introduce Hecke characters on imaginary quadratic number fields, describe our concrete family of $L$-functions in more detail, and mention various standard results that we will need later. 
\subsection{Hecke Characters in Imaginary Quadratic Number Fields}
An equivalent formulation of $\K$ having class number $1$ is that its ring of integers $\ringint_\K$ is a principal ideal domain. Explicitly, we have
\begin{align}\label{thelittleringofintegers}
\ringint_\K = \begin{cases}
\Z [\sqrt{-d}] &\text{if } d \equiv 1, 2 \text{ (mod $4$)}, \\
\Z [  (1 + \sqrt{-d})/2  ] &\text{if } d \equiv 3 \text{ (mod $4$)}.
\end{cases}
\end{align}
By using the fact that any unit in $\ringint_\K$ must have norm $1$, one may easily prove that
\begin{align*}
\ringint_\K^\times \simeq 
\begin{cases}
\Z / 2 \Z &\text{if } d = 2 \text{ or } d \geqslant 5, \\
\Z / 6 \Z &\text{if } d = 3.
\end{cases}
\end{align*}
\par Since we will later make use of the lattice structure of $\ringint_\K$, we also describe these rings in the following way.
\begin{lemma}\label{iwasawaforringint}
Under the identification $\C \simeq \R^2$, we have 
\begin{align*}
\ringint_\K = 2^{1/4} \lp \begin{matrix}
2^{-1/4} & 0 \\
0 & 2^{1/4}
\end{matrix} \rp \Z^2 ,
\end{align*}
when $d = 2$; or, when $d \geqslant 3$, 
\begin{align*}
\ringint_\K = d^{1/4} 2^{-1/2} \lp \begin{matrix}
2^{-1/2}d^{-1/4} & 0 \\
0 & 2^{1/2} d^{1/4}
\end{matrix} \rp
\lp \begin{matrix}
1 & 0 \\
1/2 & 1
\end{matrix} \rp \Z^2.
\end{align*}
\end{lemma}
\begin{proof}
When $d = 2$, this is simply a matter of expanding (\ref{thelittleringofintegers}). When $d \geqslant 3$, (\ref{thelittleringofintegers}) shows that any $\alpha \in \ringint_\K$ can be written as
\begin{align*}
-a + b/2 + b \sqrt{-d}/2 = s/2 + (s+2a)\sqrt{-d}/2
\end{align*}
with $s = -2a+b$. The decomposition of $\ringint_\K$ in this case now follows once we express $\ringint_\K$ using the variables $s$ and $a$.
\end{proof}
When $\K$ has class number $1$, a \textit{Hecke character} $\psi$ on $\K$ corresponds to a unique pair $(\chi, \chi_\infty)$ consisting of a generalized Dirichlet character $\chi$ (modulo some ideal $\mathfrak{m} \subset \ringint_\K$) and a unitary character $\chi_\infty$ on $\C^\times$. Conversely, given a pair of such characters, their (pointwise) product is a Hecke character provided that $\chi \cdot \chi_\infty$ is constant on $\ringint_\K^\times$, cf. \cite[eq. (3.80)]{iwakow}. Thus, in order for us to specify a Hecke character, it is enough to specify two characters
\begin{align*}
\chi : \big( \ringint_\K / \mathfrak{m} \big)^\times \rightarrow \C^\times, \quad \quad \chi_\infty : \C^\times \rightarrow S^1
\end{align*}
satisfying $\chi(u) \chi_\infty (u) = 1$ for all $u \in \ringint_\K^\times$. 
\par We now describe our concrete family of Hecke characters. Let $N$ be any positive integer multiple of $\left| \ringint_\K^\times \right|$. Since $\ringint_K$ is a principal ideal domain, for any $k \geqslant 1$ we can define the unitary character $\chi_{\infty, k}$ by 
\begin{align*}
\chi_{\infty, k}(I) = \chi_{\infty, k}(\alpha) = \lp \alpha /\bar{ \alpha} \rp^{N k}
\end{align*}
whenever $I = \langle \alpha \rangle$. This is well-defined since any two generators of $I$ will differ by a factor in $\ringint_K^\times$ where $\chi_{\infty, k}$ is identically equal to $1$. To make this into a Hecke character, we also need to specify a Dirichlet character that is compatible with $\chi_{\infty, k}$ in the above sense. However, we can simply take $\chi$ to be the trivial generalized Dirichlet character of modulus $\mathfrak{m} = \ringint_\K$. In this way we obtain the family of Hecke characters given by 
\begin{align*}
\psi_k( \langle \alpha \rangle) = \psi_k(\alpha) := \chi(\alpha) \cdot \chi_{\infty, k}(\alpha) = \chi_{\infty, k}(\alpha) = \lp \frac{\alpha}{\overline{\alpha}} \rp^{Nk}
\end{align*}
for $\alpha \in \ringint_\K \setminus \lbr 0 \rbr$. Since conjugation is an automorphism of $\C$, we note the relation $\overline{\psi_k} = \psi_{-k}$, which will be useful later on.\par 
In the literature it is common to write such unitary characters as $\chi_\infty (\alpha) = \lp \alpha / | \alpha | \rp^\ell$ for a suitable integer $\ell$ called the \textit{frequency}. In the case of our character $\psi_k$, we see that $\psi_k (\alpha) = \lp \alpha / |\alpha| \rp^{2Nk}$, so that $\psi_k$ has frequency $2Nk$. \par 
We note that $\psi_k$ can also be described explicitly as a function of the argument of $\alpha$, which will be convenient at certain points in the paper. Namely, if we write $\alpha = r e^{i \theta_\alpha}$, we have
\begin{align*}
\psi_k(\alpha ) = \lp \frac{r e^{i \theta_\alpha}}{r e^{- i \theta_\alpha}} \rp^{Nk} = e^{2i N k \theta_\alpha}.
\end{align*}
We also wish to speak of the "argument of the ideal $\langle \alpha \rangle$." A priori, this is not well-defined since $\langle u \alpha \rangle = \langle \alpha \rangle$ for any unit $u \in \ringint_\K^\times$. However, since any unit has argument equal to a multiple of $2 \pi / |\ringint_\K^\times |$, the effect of multiplying $\alpha$ with a unit $u$ is to change $\theta_\alpha$ by such a multiple. For this reason, by choosing $u$ appropriately, we can always ensure that the argument of $u \alpha$ lies in $\lb 0, 2 \pi / | \ringint_\K^\times | \rp$. Accordingly, the angle $\theta_{\langle \alpha \rangle}$ of the ideal $\langle \alpha \rangle$ is well-defined when taken in the interval $\lb 0, 2 \pi / | \ringint_\K^\times | \rp$.
\subsection{Hecke $L$-Functions and Their Zeros}
To each of the characters $\psi_k$ ($k \geqslant 1$) we can associate a \textit{Hecke $L$-function} given initially by the series and corresponding Euler product
\begin{align*}
L_k(s) = L_k (s, \psi_k) := \sum_{I \subset \ringint_\K \atop I \neq 0} \frac{\psi_k(I)}{\norm (I)^s} = \prod_{\mathfrak{p}} \frac{1}{1-\psi_k(\mathfrak{p})/\norm (\mathfrak{p})^s}, \quad \quad \re(s) > 1. 
\end{align*}
Let us immediately note that $L_{k} = L_{-k}$. Indeed, if $N(a,b)$ denotes the norm (see \textsc{Section} 2.3) of an element 
\begin{align*}
j(a,b) = \begin{cases}
a + i\sqrt{2}b &\text{ if } d = 2, \\
a + b(1+i \sqrt{d})/2 &\text{ if } d \geqslant 3,
\end{cases}
\end{align*}
we note that the map
$A_d = \lp \begin{smallmatrix}
1 & 0 \\
0 & -1
\end{smallmatrix} \rp 
+ \mathbbm{1}(d \geqslant 3) \cdot \lp \begin{smallmatrix}
0 & 1 \\
0 & 0
\end{smallmatrix} \rp $ 
preserves the norm $N(a,b)$ and satisfies $j \lp A_d (a,b)^\intercal \rp = \overline{j(a,b)}$.  In particular, $A_d$ defines a bijection on the set $\lbr (a,b) \in \Z^2 : N(a,b) \neq 0 \rbr$. Therefore, for $\re (s) > 1$, the trivial identity \cite[eq. (2.1)]{waxman} gives
\begin{align*}
L_{-k}(s) &= \frac{1}{\left| \ringint_\K^\times \right| }\sum_{N(a,b) \neq 0} \lp \frac{j(a,b)}{|j(a,b)|} \rp^{-2Nk} N(a,b)^{-s} \\
&= \frac{1}{\left| \ringint_\K^\times \right| }\sum_{N(a,b) \neq 0} \lp \frac{| j(a,b) |}{j(a,b)} \rp^{2Nk} N(a,b)^{-s} \\
&= \frac{1}{\left| \ringint_\K^\times \right| }\sum_{N(a,b) \neq 0} \lp \frac{\overline{j(a,b)}}{|j(a,b)|} \rp^{2Nk} N(a,b)^{-s} \\
&= \frac{1}{\left| \ringint_\K^\times \right| }\sum_{N(A_d(a,b)^\intercal ) \neq 0} \lp \frac{j\lp A_d (a,b)^\intercal \rp}{|j\lp A_d (a,b)^\intercal \rp |} \rp^{2Nk} N\lp A_d (a,b)^\intercal \rp^{-s} \\
&= \frac{1}{\left| \ringint_\K^\times \right| }\sum_{N(a,b) \neq 0} \lp \frac{j(a,b)}{|j(a,b)|} \rp^{2Nk} N(a,b)^{-s} = L_k (s).
\end{align*}
\par By a theorem of Hecke \cite[Theorem 3.8]{iwakow}, if $k \neq 0$ (so that $\psi_k$ is not the trivial character), $L_k$ admits an analytic continuation (which we will also denote by $L_k$) to the entire complex plane, and it satisfies the functional equation
\begin{align*}
\Lambda \lp s, \psi_k \rp = \frac{\tau \lp \psi_k \rp}{i^\ell \sqrt{\norm (\mathfrak{m})}} \Lambda \lp 1-s, \overline{\psi_k} \rp ,
\end{align*}
where $\mathfrak{m} = \ringint_\K$ is the modulus of $\psi_k$, $\ell = 2Nk$ is the frequency, and $\Lambda \lp s, \psi_k \rp$ denotes the \textit{completed $L$-function}
\begin{align*}
\Lambda \lp s, \psi_k \rp = \Lambda_k (s) :=  L_k (s) \frac{\big( |D| \norm (\mathfrak{m}) \big)^{s/2}}{(2 \pi)^s} \Gamma \big( s + |\ell|/2 \big),
\end{align*}
and where $\tau (\psi_k)$ denotes the Gauss sum
\begin{align*}
\tau \lp \psi_k \rp = \psi_k (\gamma) \psi_k (\mathfrak{c})^{-1} \sum_{\alpha \in \mathfrak{c}/\mathfrak{c} \mathfrak{m}} \exp \lp 2 \pi i \mathrm{Tr}\lp \alpha / \gamma \rp \rp,
\end{align*}
cf. \cite[eq. (3.86)]{iwakow}. Here $\gamma \in \ringint_K$ and $\mathfrak{c} \subset \ringint_K$ are arbitrary except for the requirements that $\mathfrak{c}$ should be an ideal, and that $\gamma$ and $\mathfrak{c}$ should satisfy $(\mathfrak{c}, \mathfrak{m}) = 1$ and $\mathfrak{c} \mathfrak{d} \mathfrak{m} = \langle \gamma \rangle$, where $\mathfrak{d}$ is the \textit{different} of $K$. In our case, we have $\tau (\psi_k) = 1$: Since we have $\mathfrak{m} = \ringint_K$ and $\mathfrak{d} = \langle \sqrt{D} \rangle$, these conditions are satisfied with $\mathfrak{c} = \ringint_K$ and $\gamma = \sqrt{D}$. In combination with the relation $\overline{\psi_k} = \psi_{-k}$ and the fact that the frequency $\ell$ of $\psi_k$ is $2Nk \equiv 0$ (mod $4$), this means that the root number of $L_k$ is $1$, and the functional equation assumes the simpler form
\begin{align}\label{completedfcteq}
\Lambda_k (s) = \Lambda_{-k} (1-s) = \Lambda_k (1-s),
\end{align}
where also the completed $L$-function can be described in the simpler form
\begin{align}\label{completedlfct}
\Lambda_k (s) =  L_k (s) \frac{|D|^{s/2}}{(2 \pi)^s} \Gamma \big( s + Nk \big).
\end{align}
Of course, it is also possible to recast the identity (\ref{completedfcteq}) as a statement about $L_k$ that does not explicitly involve $\Lambda_k$. Doing so, we find that
\begin{align}\label{Lfcteq}
L_k \lp s \rp = L_k(1-s) X_k(s),
\end{align}
where 
\begin{align}\label{fcteqXsimple}
X_k (s) :&= \frac{\Gamma (1-s+Nk)}{\Gamma (s+Nk)} |D|^{1/2 - s} (2 \pi)^{2s-1}.
\end{align}
\par If $K \geqslant 1$ is an integer, then as we mentioned in the introduction, we will use the notation 
\begin{align*}
\mathcal{F}(K) := \lbr L_k : 1 \leqslant k \leqslant K \rbr
\end{align*}
to denote our family of $L$-functions. We wish to normalize the zeros of this family so that they have mean spacing $1$. This of course warrants an understanding of the asymptotic number $N_k(T)$ of zeros of $L_k$ in the critical strip $0 \leqslant \re (z) \leqslant 1$ up to a given height $T$ as $k \rightarrow \infty$. Estimates of the count $N_k(T)$ for general $L$-functions are abundant in the literature, but these usually provide an asymptotic as $T \rightarrow \infty$. We therefore prove the following result, which gives an asymptotic expression for $N_k(T)$ when it is not the height, but rather the size of our family that tends to infinity.
\begin{prop}\label{countzeros}
Let $k \geqslant 1$, and assume the Riemann Hypothesis for $L_k$.  For $T > 0$, let 
\begin{align*}
N_k(T) := \# \Big\{ z \in \C : L_k(z) = 0, \, 0 \leqslant \re (z) \leqslant 1, \, -T \leqslant \im (z) \leqslant T \Big\}
\end{align*}
be the number of zeros of $L_k$ in the critical strip up to absolute height $T$. Then as $k \rightarrow \infty$, 
\begin{align*}
N_k(T) \sim \frac{2 T \log k}{\pi}.
\end{align*}
\end{prop}
\begin{proof}
Let us write $X := 1/2 + Nk$. By \cite[Thm. 5]{carneiro} and \cite[Eq. (4.1)]{carneiro}, the main term of $N_k(T)$ comes from the integral 
\begin{align*}
\frac{1}{\pi} \int_{-T}^T \re\lp \frac{\Gamma'(X + it)}{\Gamma(X + it)} \rp \hsp \dd t = \frac{1}{\pi} \re \lp \int_{-T}^T \frac{\Gamma'(X + it)}{\Gamma(X + it)} \hsp \dd t \rp.
\end{align*}
Indeed, in the notation of that paper, we have 
\begin{align*}
L(z, \pi_\infty) = \frac{|D|^{z/2}}{(2 \pi)^z} \Gamma \big( z + Nk \big);
\end{align*}
and since the logarithmic derivative of the factor $|D|^{z/2}(2 \pi)^{-z}$ is constant, the integral of the logarithmic derivative of $\left| D \right|^{(X+it)/2}(2 \pi)^{-X-it}$ over the line $-T \leqslant t \leqslant T$ is at most a constant times $T$. \par 
Since $-i \log \Gamma (X+it)$ is a primitive function for the logarithmic derivative in the integrand above, and since our domain of integration lies in the right half-plane, we can use Stirling's formula $\log \Gamma(z) = z \log z - z - (\log z)/2 + O\lp \max \lbr 1, 1/z \rbr \rp$ to obtain
\begin{align*}
\int_{-T}^T \frac{\Gamma'(X + it)}{\Gamma(X - it)} \hsp \dd t
&= -i \cdot X \log \frac{X+iT}{X-iT} + T \log \lp X^2 + T^2 \rp \\
&\quad \quad - 2T + i \cdot \frac{1}{2} \log \frac{X+iT}{X-iT} + O\lp 1 \rp.
\end{align*}
By taking the real part and using that $T \log \lp X^2 + T^2 \rp = 2 T \log k + O_T (1)$, we therefore have
\begin{align*}
\frac{1}{\pi} \int_{-T}^T \re\lp \frac{\Gamma'(X + it)}{\Gamma(X + it)} \rp \hsp \dd t = \frac{2T \log k}{\pi} + O_T(1)
\end{align*}
as $k \rightarrow \infty$. This concludes the proof. 
\end{proof}
\subsection{Splitting Behaviour of Rational Primes}
We now record some facts about how a prime $p \in \Z$ behaves in the extension $\K = \Q (\sqrt{-d})$, where we recall that $d$ is one of the integers $2, 3, 7, 11, 19, 43,$ $67, 163$. For a rational prime $p$, we say that $p$ is \textit{ramified} if $p = q^2 u$ for a prime $q \in \ringint_\K$ and $u \in \ringint_\K^\times$, \textit{split} if $p = q \overline{q} u$ for a prime $q \in \ringint_\K$ and $u \in \ringint_\K^\times$, or \textit{inert} if $p$ is a prime element in $\ringint_\K$. Any rational prime belongs to exactly one of these categories. \par 
By using the norm map $\ringint_\K \rightarrow \Z_+$, we immediately see that 
if $p$ is ramified or split, then some prime element $q \in \ringint_\K$ has norm $p$. Thus, we can reduce certain aspects of the study of the splitting behaviour in $\K$ of rational primes to studying which primes are represented over $\Z$ by the quadratic forms defined by the norm. We now describe these forms in more detail. \par 
If $d \equiv 3$ (mod $4$), any integer in $\K$ has the form $a + b(1 + \sqrt{-d})/2$, and its norm is 
\begin{align}\label{1mod4-norm}
\norm \lp a + \frac{b(1 + \sqrt{-d})}{2} \rp = \lp a + \frac{b}{2} \rp^2 + \frac{db^2}{4} = a^2 + ab + \frac{d+1}{4}b^2.
\end{align}
Note that this form has integer coefficients precisely because of the congruence condition on $d$. On the other hand, if $d \equiv 1, 2$ (mod $4$), then any integer in $\K$ has the form $a + b \sqrt{-d}$, and its norm is
\begin{align}\label{23mod4-norm}
\norm \lp a + b \sqrt{-d} \rp = a^2 + d b^2.
\end{align}
\par 
It is well-known that a rational prime $p$ is ramified in $\K$ if and only if $p \mid D$, the discriminant of $\K$. Since 
\begin{align*}
|D| = \begin{cases}
d &\text{ if } d \equiv 3 \text{ (mod $4$)}, \\
4d &\text{ if } d \equiv 1,2 \text{ (mod $4$)},
\end{cases}
\end{align*}
we see that $p$ is ramified in $\K$ if and only if $p = d$. \par 
By \cite[Prop I.8.5]{neukirch} we know that if $p \neq 2$, then $p$ splits in $\K$ if and only if the Legendre symbol $(-d/p)$ equals $1$. On the other hand, if $p = 2$ we can use (\ref{1mod4-norm}) to see that $p$ only splits in $\K$ when $d = 7$. (This also implies that whenever $d \neq 2, 7$, the prime $2$ is inert in $\ringint_\K$.) Indeed, if $d \equiv 3$ (mod $4$) and $2$ splits in $\K$, there are rational integers $a$ and $b$ such that $q = a + b \lp 1 + \sqrt{-d} \rp/2$ and
\begin{align*}
\norm (q) - 2 = a^2 + ab + \frac{d+1}{4} b^2 - 2 = 0,
\end{align*}
which implies (by solving the quadratic equation in $a$) that $8 - db^2$ has to be a square number. Given that $d \equiv 3$ (mod $4$), this is only possible if $d = 7$ and $b = \pm 1$. This shows that the element $q = (1 + \sqrt{-7})/2$ has norm $2$. It is clear that $2 = q \overline{q}$ and that $q$ and $\overline{q}$ are not equivalent up to multiplication by units. \par 
In the remainder of the paper, we will also need to work with the "Legendre symbol modulo composite numbers." We therefore introduce the Kronecker symbol, which will also be convenient for the purpose of expressing the splitting behavior of a prime. If $n = p_1^{e_1} \cdots p_k^{e_k}$ is (the prime factorization of) a positive integer and $a \in \Z$, the \textit{Kronecker symbol} $(a/n)$ is defined as
\begin{align*}
\lp \frac{a}{n} \rp = \prod_{i = 1}^k \lp \frac{a}{p_i} \rp^{e_i},
\end{align*}
where the symbol $(a / p_i )$ appearing on the right-hand side is the Legendre symbol if $p_i$ is an odd prime, or otherwise given by 
\begin{align*}
\lp \frac{a}{2} \rp = \begin{cases}
0 &\text{ if $a$ is even,}\\
1 &\text{ if $a \equiv \pm 1$ (mod $8$),}\\
-1 &\text{ if $a \equiv \pm 3$ (mod $8$).}
\end{cases}
\end{align*}
We note that if $a \not \equiv 3$ (mod $4$), then the map $n \mapsto (a/n)$ is a quadratic Dirichlet character, cf. \cite[§5]{davenport}. In particular, for $d$ equal to any of the eight non-trivial Heegner numbers that we are considering, $\chi(n) := (-d/n)$ is a Dirichlet character. \\ \par 
We can now sum up the above discussion about the splitting behavior of rational primes in the ring $\ringint_\K$ in the following lemma. 
\begin{lemma}\label{splittingbehaviour}
Let $p$ be a rational prime. In the number field $\K$,
\begin{align*}
p \text{ is } \begin{cases}
\text{ramified} &\text{if } \chi(p) = 0 \text{ (that is, $p = d$)}, \\
\text{split} &\text{if } \chi(p) = 1, \\
\text{inert} &\text{if } \chi(p) = -1.
\end{cases}
\end{align*}
\end{lemma}
We end this section with the following elementary lemma whose proof we include for the sake of completeness.
\begin{lemma}\label{standardalgnt}
Let $\mathfrak{p} \subset \ringint_\K$ be a prime ideal. Then $\mathfrak{p}$ lies over a rational prime $p$ (that is, $\mathfrak{p} \cap \Z = p \Z$) if and only if $\mathfrak{p} \mid p \ringint_\K$. Hence, the prime ideals with norm equal to a power of $p$ are precisely the prime ideals dividing $p \ringint_\K$.
\end{lemma}
\begin{proof}
Suppose that $\mathfrak{p}$ lies over $p$. Then clearly $p \in \mathfrak{p}$, and by the ideal property of $\mathfrak{p}$ we therefore have $p \ringint_\K \subset \mathfrak{p}$, which means $ \mathfrak{p} \mid p \ringint_\K$. \par 
Conversely, if $\mathfrak{p} \mid p \ringint_\K$, then $\mathfrak{p} \supset p \ringint_\K$, and hence $\mathfrak{p} \cap \Z \supset p \ringint_\K \cap \Z \supset p \Z$. Since $p \Z$ is a maximal ideal in $\Z$ and $\mathfrak{p} \cap \Z \subset \Z$ must be a proper ideal, it follows that $p \Z = \mathfrak{p} \cap \Z$. Therefore $\mathfrak{p}$ lies over $p$.
\end{proof}
\section{Implications of the Ratios Conjecture}
In this section, we describe the $L$-functions Ratios Conjecture which is due to Conrey, Farmer, and Zirnbauer (\cite{confarzir}), generalizing a conjecture of Farmer about the Riemann zeta function (see \cite{farmer}). We show how the conjecture implies strong estimates for the $1$-level density of the zeros of the family $\family (K)$ as $K \rightarrow \infty$. \par 
The \textit{Ratios Conjecture} states that a sum of ratios of (products of) $L$-functions evaluated at certain parameters should obey a specific asymptotic estimate. We now describe the recipe from \cite{confarzir} for conjecturing such an asymptotic. However, for the sake of simplicity, we do not describe the most general case possible. In the case of two $L$-functions in the numerator and denominator, one considers
\begin{align*}
Q(s, \bm{\alpha}, \bm{\gamma}; \chi) = \frac{L(s + \alpha_1, \chi) L (s + \alpha_2, \chi)}{L(s + \gamma_1, \chi) L(s + \gamma_2, \chi)},
\end{align*}
where $L(s, \chi)$ is the $L$-function associated to a character $\chi$ and satisfying the functional equation
\begin{align*}
L(s, \chi) = X(s, \chi) L(1-s, \overline{\chi}).
\end{align*}
The recipe is as follows: 
\begin{itemize}
    \item \textit{Approximate functional equation for $L$} \\ Replace each $L$-function in the numerator of $Q(s, \bm{\alpha}, \bm{\gamma}; \chi)$ with the two main terms from its approximate functional equation, completely disregarding the remainder term. 
    \item \textit{Infinite series for $1/L$}\\ 
    Replace each reciprocal $L$-function with its expression as an infinite series involving a suitable Möbius function. 
    \item \textit{Extend ranges and regroup factors}\\
    Extending the ranges of all series to infinity and multiplying out the resulting expression, write each of the four resulting terms as 
    \begin{align*}
    \big( \text{product of root numbers $\varepsilon_\chi$} \big) \times \big( \text{product of $X(\cdot, \chi)$-factors} \big) \times \big( \text{sum over $n_1, n_2, n_3, n_4$} \big),
    \end{align*}
    where $n_1, n_2, n_3, n_4$ are the indexing variables from the two approximate functional equations and the two infinite series giving the reciprocal $L$-functions. 
    \item \textit{Average factors over the family} \\
    Replace each product of root numbers, each product of $X(\cdot, \chi)$-factors, and each summand in the last factor with their respective averages over the family $\mathcal{F} = \lbr \chi \rbr$. Denote the resulting expression by $M(s, \bm{\alpha}, \bm{\gamma})$.
    \item \textit{Statement of the conjecture}\\
    The conjecture now states that for any $\varepsilon > 0$,
    \begin{align*}
    \sum_{\chi \in \mathcal{F}} Q(s, \bm{\alpha}, \bm{\gamma}; \chi)w \lp q(\chi) \rp &= \lp 1 + O \lp e^{(-1/2 + \varepsilon) q ( \chi)} \rp \rp \sum_{\chi \in \mathcal{F}} M(s, \bm{\alpha}, \bm{\gamma}) w \lp q(\chi) \rp,
    \end{align*} 
    where $q (\chi) := \left| X'(1/2, \chi) \right|$ denotes the \textit{log conductor} of $\chi$, and $w$ is a suitable weight function.
\end{itemize}
\par A later addition to the Ratios Conjecture \cite[Sect. 2]{consnaith} is that such an asymptotic is expected to hold provided that $\alpha$ and $\gamma$ satisfy the constraints
\begin{align}\label{conjcondsforalphagamma}
-\frac{1}{4} < \mathrm{Re}(\alpha) < \frac{1}{4}, \quad \frac{1}{\log K} \ll \mathrm{Re}(\gamma) < \frac{1}{4}, \quad \mathrm{Im}(\alpha), \, \mathrm{Im}(\gamma) \ll_\varepsilon K^{1 - \varepsilon}.
\end{align}
\subsection{Ingredients for the Recipe}
To follow the recipe outlined above in the case of the expression
\begin{align}\label{expressionofinterest}
R_K \lp \alpha, \gamma \rp = \frac{1}{K} \sum_{k = 1}^K \frac{L_k (1/2 + \alpha)}{L_k (1/2 + \gamma)},
\end{align}
we first need to describe the approximate functional equation for $L_k$ and obtain an expression for the reciprocal $L_k^{-1}$ as an infinite series. \par 
We begin by describing the approximate functional equation, which involves writing $L_k$ in a different way that is more reminiscent of classical Dirichlet $L$-functions. To this end, we observe that
\begin{align*}
L_k (s) = \sum_{I \subset \ringint \atop I \neq 0} \frac{\psi_k (I)}{\norm (I)^s} 
= \sum_{n \geqslant 1} \lp \sum_{\norm (I) = n} \psi_k (I) \rp n^{-s} 
= \sum_{n \geqslant 1} \frac{A_k (n)}{n^s} 
\end{align*}
for $\re (s) > 1$, with 
\begin{align*}
A_k (n) := \sum_{\norm (I) = n} \psi_k (I), \quad \quad n \geqslant 1.
\end{align*}
We note that $A_k$ is real-valued. Indeed, this follows from the definition of $\psi_k$ and the fact that for any $\alpha \in \ringint_\K$,
\begin{align*}
\norm \lp \langle \alpha \rangle \rp = \alpha \overline{\alpha} = \norm \lp \langle \overline{\alpha} \rangle \rp.
\end{align*}
Hence, any ideal that contributes to the sum defining $A_k (n)$ is accompanied by its conjugate ideal, provided that these are different. (If they are not, then their contribution to $A_k$ is clearly real.) Complex conjugation of $A_k$ therefore amounts to permuting the terms in the sum, which of course leaves $A_k$ unchanged.
\par Using the series defining $L_k$ and the functional equation (\ref{Lfcteq}), we can now describe the approximate functional equation of $L_k$ as
\begin{align}\label{approxfcteq}
\nonumber L \lp s, \psi_k \rp &\approx \sum_{n < x} \frac{A_k (n)}{n^s} + X_k (s) \sum_{n < y} \frac{\overline{A_k (n)}}{n^{1-s}} \\
&= \sum_{n < x} \frac{A_k (n)}{n^s} + X_k (s) \sum_{n < y} \frac{A_k (n)}{n^{1-s}},
\end{align}
where $x$ and $y$ are positive real parameters.\par 
As mentioned above, the recipe of the Ratios Conjecture also requires us to obtain a formula for the reciprocal function $L_k^{-1}$. Taking the reciprocal of the Euler product and using the fact that $\psi_k$ and the ideal norm are completely multiplicative, we see that
\begin{align*}
\frac{1}{L_k (s)} 
&= \prod_{\mathfrak{p}} \lp 1 - \frac{\psi_k (I)}{\norm(I)^s} \rp \\
&= 1 - \sum_{n = \mathfrak{p}_1} \frac{\psi_k \lp \mathfrak{p}_1 \rp }{\norm \lp \mathfrak{p}_1 \rp^s}  +
\sum_{n = \mathfrak{p}_1 \mathfrak{p}_2} \frac{\psi_k \lp \mathfrak{p}_1 \mathfrak{p}_2 \rp }{\norm \lp \mathfrak{p}_1 \mathfrak{p}_2 \rp^s} 
- \sum_{n = \mathfrak{p}_1 \mathfrak{p}_2 \fracp_3} \frac{\psi_k \lp \mathfrak{p}_1 \mathfrak{p}_2 \fracp_3 \rp }{\norm\lp \mathfrak{p}_1 \mathfrak{p}_2 \fracp_3 \rp^s} + \cdots \\
&= \sum_{I \subset \ringint \atop I \neq 0} \frac{\mu (I) \psi_k (I)}{\norm(I)^s},
\end{align*}
where 
\begin{align*}
\mu(I) = \begin{cases}
(-1)^{n} &\text{if } I \text{ is the product of $n$ distinct primes},\\
0 &\text{otherwise},
\end{cases}
\end{align*}
is the natural analogue of the Möbius function. With
\begin{align*}
\mu_k (n) := \sum_{\norm(I) = n} \mu (I) \psi_k (I),
\end{align*}
we therefore obtain the formula
\begin{align}\label{lfctreciprocal}
\frac{1}{L_k (s)} = \sum_{n \geqslant 1} \frac{\mu_k (n)}{n^s}, \quad \quad \re(s) > 1.
\end{align}
\begin{lemma}\label{muAmultiplicative}
For any $k \geqslant 1$, the functions $\mu_k$ and $A_k$ are multiplicative.
\end{lemma}
\begin{proof}
It suffices to show that if $p \neq q$ are rational primes and $a, b \geqslant 1$ are rational integers, then 
\begin{align*}
\mu_k \lp p^a q^b \rp = \mu_k \lp p^a \rp \mu_k \lp q^b \rp, \quad \quad A_k \lp p^a q^b \rp = A_k \lp p^a \rp A_k \lp q^b \rp.
\end{align*}
We prove the claim for $\mu_k$ as the proof of the claim for $A_k$ is analogous. \par 
Since $\psi_k$ is completely multiplicative, we have
\begin{align*}
\mu_k \lp p^a \rp \mu_k \lp q^b \rp = \sum_{\norm\lp I_1 \rp = p^a  \atop \norm\lp I_2 \rp = q^b} \mu \lp I_1 I_2 \rp \psi_k \lp I_1 I_2 \rp. 
\end{align*}
It remains to show that all ideals $I$ of norm $p^a q^b$ have the form $I_1 I_2$ for ideals $I_1$ and $I_2$ of norm $p^a$ and $q^b$, respectively. However, this follows immediately from the unique factorization of ideals in $\ringint_\K$ in combination with the properties of the norm map.
\end{proof}
\subsection{Following the Recipe}
In accordance with the recipe outlined above, we now use (\ref{approxfcteq}) and (\ref{lfctreciprocal}) and compute
\begin{align}\label{RC1-3}
\frac{L_k (1/2 + \alpha)}{L_k (1/2 + \gamma)} \approx \sum_{m, n \geqslant 1} \frac{\mu_k (m)A_k (n)}{m^{1/2 + \gamma}n^{1/2 + \alpha}} + X_k (1/2 + \alpha) \sum_{m, n \geqslant 1} \frac{\mu_k(m) A_k(n)}{m^{1/2 + \gamma} n^{1/2 - \alpha}},
\end{align}
where we extended the ranges of summation to infinity in accordance with the recipe provided above. We now average the individual factors. We begin with the following lemma.
\begin{lemma}\label{waxmanslemma3.1}
As $K \rightarrow \infty$, we have
\begin{align}\label{Xkaverage}
\begin{split}
\frac{1}{K} \sum_{k = 1}^K X_k (1/2 + \alpha) &= \frac{1}{1-2 \alpha}  \lp \frac{2 \pi}{KN \sqrt{|D|}} \rp^{2 \alpha} + O_\alpha \lp \frac{1}{K} \rp + O_\alpha \lp \frac{1}{K^{1 + 2 \alpha}} \rp.
\end{split}
\end{align}
\end{lemma}
\begin{proof}
We can argue exactly as in the proof of \cite[Lemma 3.1]{waxman}. The only thing we need to take into account is that we have $\textbf{k} = 1/2 + Nk$ in the notation of that paper.
\end{proof}
It remains for us to compute the average of the summands in (\ref{RC1-3}). Here it is fruitful to rewrite the sums appearing there as products in order to take advantage of the multiplicative nature of the function $\mu_k$. To that end, with the help of Lemma \ref{muAmultiplicative} we note that
\begin{align}\label{rewriteasproduct}
\sum_{m ,n \geqslant 1} \frac{\mu_k (m)A_k (n)}{m^{1/2 + \gamma}n^{1/2 + \alpha}} = \prod_{p \atop \text{ prime}} \sum_{m,n \geqslant 0} \frac{\mu_k \lp p^m \rp A_k \lp p^n \rp}{p^{m(1/2 + \gamma)}p^{n(1/2 + \alpha)}}.
\end{align}
\par At this point, we will describe the values of $\mu_k$ and $A_k$ on prime powers more precisely.
\begin{lemma}\label{muonprimepowers}
Let $p$ be a rational prime. Then 
\begin{align*}
\mu_k \lp p^m \rp = 
\begin{cases}
1 &\text{if } m = 0, \\
-A_k (p) &\text{if } m = 1,\\
-1 &\text{if $m = 2$ and $p$ is inert,}\\
1 &\text{if $m = 2$ and $p$ splits,}\\
0 &\text{otherwise}.
\end{cases}
\end{align*}
\end{lemma}
\begin{proof}
The first two claims follow immediately from the definition of $\mu_k$. As for the third claim, we have $\mu_k \lp p^2 \rp = -1$ since only the prime ideal $\langle p \rangle$ has norm $p^2$ by Lemma \ref{standardalgnt}. \par 
Turning to the fourth claim, let us suppose that $\langle p \rangle$ has prime factors $\mathfrak{p}_1$ and $\mathfrak{p}_2$. In this case, there is only one squarefree ideal of norm $p^2$, namely $\langle p \rangle$, since only the prime ideals dividing $\langle p \rangle$ have norms equal to a power of $p$ by Lemma \ref{standardalgnt}. \par 
We now prove the fifth and final claim. We first examine the case where $m = 2$ and $p$ is ramified in $\ringint_\K$ with $\langle p \rangle = \mathfrak{q}^2$ for some prime ideal $\mathfrak{q}$. This immediately implies that $\mu_k (p^2) = 0$ since the only ideal of norm $p^2$ is $\mathfrak{q}^2$ by Lemma \ref{standardalgnt}. For the final case, we assume $m \geqslant 3$. To prove the claim, it is enough to show that the norm of a product of distinct prime ideals lying over $p$ is at most $p^2$. However, this is clear from Lemma \ref{standardalgnt} once we consider the three possible splitting behaviours of $p$. 
\end{proof}
We also need the following description of the function $A_k$ on prime powers.
\begin{lemma}\label{Akonprimepowers}
$A_k$ assumes the following values on prime powers (where we understand that the case $n = 0$ takes precedence over the remaining cases):
\begin{align*}
A_k (p^n) = 
\begin{cases}
1 &\text{if } n = 0, \\[+0.5em]
(q^n / \, \overline{q}^n)^{Nk} &\text{if $\langle p \rangle = \langle q \rangle^2$,}\\[+0.5em]
1 &\text{if } p \text{ is inert and } n \text{ is even}, \\[+0.5em]
0 &\text{if } p \text{ is inert and } n \text{ is odd}, \\[+0.5em]
\sssum_{j = -n/2}^{n/2} \psi_k (q)^{2j} &\text{if $\langle p \rangle = \langle q \rangle \langle \overline{q} \rangle$ and $n$ is even},\\[+0.5em]
\sssum_{j = -(n+1)/2}^{(n-1)/2} \psi_k (q)^{2j+1} &\text{if $\langle p \rangle = \langle q \rangle \langle \overline{q} \rangle$ and $n$ is odd}.
\end{cases}
\end{align*}
\end{lemma}
\begin{proof}
Since only $\langle 1 \rangle = \ringint_\K$ has norm $1$, the first claim follows immediately from the definition. In the following we suppose that $n \geqslant 1$. \par 
Suppose that $p$ is ramified in $\ringint_\K$ with $\langle p \rangle = \mathfrak{q}^2$. Then Lemma \ref{standardalgnt} immediately implies that the only ideal of norm $p^n$ is $\mathfrak{q}^n$, which proves the claim. \par 
Suppose that $p$ is inert. Then there is no prime ideal of norm $p$. Moreover, the ideal $\langle p \rangle$ is prime and has norm $p^2$, and it is the only such prime ideal. It follows that if $I$ is any ideal of norm $p^n$, then $n$ must be even, and in this case $I = \langle p \rangle^{n/2} = \langle p^{n/2} \rangle$. This yields the third and fourth claims.
\par Turning to the final claim, we assume that $p$ splits as $q \overline{q}$ in $\ringint_\K$. Then clearly $\langle q \rangle$ and $\langle \overline{q} \rangle$ are the only prime ideals of norm $p$. Moreover, as we assumed that $p$ is not ramified, it is clear that $\langle q \rangle \neq \langle \overline{q} \rangle$. Similarly as before, we also see that no prime ideals of norm $p^2$ can exist. It follows that if $I$ has norm $p^n$, then $I = \langle q \rangle^{j} \langle \overline{q} \rangle^{n-j} = \langle q^j \overline{q}^{n-j} \rangle$ for some $j = 0, \ldots, n$. Thus, 
\begin{align*}
A_k (p^n) = \sum_{j = 0}^n \lp \frac{q^j \overline{q}^{n-j}}{\overline{q}^j q^{n-j}} \rp^{Nk} = \sum_{j = 0}^n \lp \frac{q^{2j-n}}{\overline{q}^{2j-n}} \rp^{Nk},
\end{align*}
and therefore
\begin{align*}
A_k (p^n) =
\begin{cases}
\sum_{j = -n/2}^{n/2} \psi_k (q)^{2j} &\text{if $n$ is even},\\[+1.5em]
\sum_{j = -(n+1)/2}^{(n-1)/2} \psi_k (q)^{2j+1} &\text{if $n$ is odd}.
\end{cases}
\end{align*}
This concludes the proof.
\end{proof}
In accordance with the Ratios Conjecture, we now compute the (asymptotic) averages of the function $\mu_k (p^m) A_k (p^n)$ as $k = 1, \ldots, K$. We will denote this by $\delta_p (m,n)$ so that
\begin{align*}
\delta_p (m,n) = \lim_{K \rightarrow \infty} \frac{1}{K} \sum_{k = 1}^K \mu_k (p^m) A_k (p^n).
\end{align*}
The existence of this limit will be clear from the consideration of the special cases of $p$ (split, inert, or ramified). In anticipation of this, we furthermore write
\begin{align*}
\delta_p (m,n) = 
\begin{cases}
\delta_\mathrm{in}(m,n) &\text{if $p$ is inert}\\
\delta_\mathrm{sp}(m,n) &\text{if $p$ splits}\\
\delta_\mathrm{ram}(m,n) &\text{if $p$ is ramified}.
\end{cases}
\end{align*}
\begin{lemma}\label{limitdescription}
We have
\begin{align*}
\delta_\mathrm{in}(m,n) &= 
\begin{cases}
0 &\text{if } m = 1 \text{ or } m \geqslant 3 \text{ or } n \text{ is odd,}\\
1 &\text{if } m = 0 \text{ (and $n$ is even),}\\
-1 &\text{if } m = 2 \text{ (and $n$ is even),}
\end{cases} \\
\delta_\mathrm{sp}(m,n) &= 
\begin{cases}
1 &\text{if $m = 0$ and $n$ is even,}\\
-2 &\text{if $m = 1$ and $n$ is odd,}\\
1 &\text{if $m = 2$ and $n$ is even,}\\
0 &\text{otherwise,}\\
\end{cases} \\
\delta_\mathrm{ram}(m,n) &= 
\begin{cases}
1 &\text{if $m = 0$,}\\ 
-1 &\text{if $m = 1$,}\\
0 &\text{if $m \geqslant 2$.}
\end{cases}
\end{align*}
\end{lemma}
\begin{proof}
This follows immediately from Lemma \ref{muonprimepowers} and Lemma \ref{Akonprimepowers}. In the case of a ramified prime $p$, we also use the fact that $A_k (p^n) = \lp q^n / \, \overline{q}^n \rp^{Nk} = 1$, since in this case $q/ \, \overline{q} \in \ringint_\K^\times$. 
\end{proof}
Using Lemma \ref{limitdescription}, we can describe the limiting average of the right-hand side of (\ref{rewriteasproduct}) as follows. If the prime $p$ is inert, we write
\begin{align*}
G_\mathrm{in}(p; \alpha, \gamma) := \sum_{m,n \geqslant 0} \frac{\delta_\mathrm{in}(m,n)}{p^{m(1/2 + \gamma) + n(1/2 + \alpha)}},
\end{align*}
and we define $G_\mathrm{sp}(p; \alpha, \gamma)$ and $G_\mathrm{ram}(p; \alpha, \gamma)$ analogously. By using Lemma \ref{limitdescription}, we then deduce that 
\begin{align*}
\allowdisplaybreaks
G_\mathrm{in}(p ; \alpha, \gamma) &= \sum_{n \geqslant 0} p^{-n(1+2\alpha)} - p^{-(1+2 \gamma)} \sum_{n \geqslant 0} p^{-n(1 + 2 \alpha)} = \lp 1 - p^{-(1+2 \gamma)} \rp \lp 1 - p^{-(1+2 \alpha)} \rp^{-1},  \\[+1em]
G_\mathrm{sp}(p ; \alpha, \gamma) &= \lp 1 - 2p^{-(1 + \alpha + \gamma)} + p^{-(1 + 2 \gamma)} \rp \lp 1 - p^{-(1+2 \alpha)} \rp^{-1}, \\[+1.7em]
G_\mathrm{ram}(p; \alpha, \gamma) &= \sum_{n \geqslant 0} \lp p^{-n(1/2 + \alpha)} - p^{-(1/2 + \gamma)} p^{-n(1/2 + \alpha)} \rp = \lp 1 - p^{-(1/2 + \gamma)} \rp \lp 1 - p^{-(1/2 + \alpha) } \rp^{-1}.
\end{align*}
Moreover, we let
\begin{align*}
F_2 (\alpha, \gamma) := \begin{cases}
1, &\text{if } d = 2, \\
\lp 1 - 2^{-(\alpha + \gamma)} + 2^{-(1 + 2 \gamma)} \rp \lp 1 - 2^{-(1+2 \alpha)} \rp^{-1} &\text{if } d = 7, \\
\lp 1 - 2^{-(1+2 \gamma)} \rp \lp 1 - 2^{-(1+2 \alpha)} \rp^{-1} &\text{otherwise}. 
\end{cases}
\end{align*}
We now see that the product of $G_* (p ; \alpha, \gamma)$ over all rational primes equals
\begin{align}\label{whatamess}
\begin{split}
G(\alpha, \gamma) 
&= F_2 (\alpha, \gamma) \, G_\mathrm{ram}\lp d; \alpha, \gamma \rp  \prod_{p \geqslant 3 \atop (-d/p) = 1} G_\mathrm{sp}(p; \alpha, \gamma) \prod_{p \geqslant 3 \atop (-d/p) = -1} G_\mathrm{in}(p ; \alpha, \gamma) \\
&= \lp 1 - d^{-(1/2 + \gamma)} \rp \lp 1 - d^{-(1/2 + \alpha)} \rp^{-1} F_2(\alpha, \gamma) \prod_{p \geqslant 3 \atop (-d/p) = 1} \lp 1 - 2p^{-(1 + \alpha + \gamma)} + p^{-(1 + 2 \gamma)} \rp \\ 
 &\quad \quad \times \prod_{p \geqslant 3 \atop (-d/p) = -1}  \lp 1 - p^{-(1+2 \gamma)} \rp \prod_{p \geqslant 3 \atop p \neq d} \lp 1 - p^{-(1+2 \alpha)} \rp^{-1} \\
 &= \lp 1 - \mathbbm{1}(d \geqslant 3) \cdot d^{-(1+2 \alpha)} \rp \lp 1 - d^{-(1/2 + \gamma)} \rp \lp 1 - d^{-(1/2 + \alpha)} \rp^{-1} F_2(\alpha, \gamma) \\[+1.3em] 
 &\quad \quad \times  \prod_{p \geqslant 3 \atop (-d/p) = 1} \lp 1 - 2p^{-(1 + \alpha + \gamma)} + p^{-(1 + 2 \gamma)} \rp  \prod_{p \geqslant 3 \atop (-d/p) = -1}  \lp 1 - p^{-(1+2 \gamma)} \rp \\[+0.5em]
 &\quad \quad \times \hspace{1em} \prod_{p \geqslant 3} \lp 1 - p^{-(1+2 \alpha)} \rp^{-1} \\[+0.5em]
 &=  \lp 1 + d^{-(1/2 + \alpha)} \rp \lp 1 - d^{-(1/2 + \gamma)} \rp \tilde{F}_2 (\alpha, \gamma) \prod_{p \geqslant 3 \atop (-d/p) = 1} \lp 1 - 2p^{-(1 + \alpha + \gamma)} + p^{-(1 + 2 \gamma)} \rp \\ 
&\quad \quad \times \prod_{p \geqslant 3 \atop (-d/p) = -1}  \lp 1 - p^{-(1+2 \gamma)} \rp \zeta(1 + 2 \alpha),
\end{split}
\end{align}
when $\re (\alpha ) > 0$. Here $\tilde{F}_2$ is the function obtained by possibly removing the factor $\lp 1-2^{-(1+ 2 \alpha)} \rp^{-1}$ from $F_2$, i.e. 
\begin{align*}
\tilde{F}_2 (\alpha, \gamma) := \begin{cases}
1 &\text{if } d = 2, \\
1 - 2^{-(\alpha + \gamma)} + 2^{-(1 + 2 \gamma)} &\text{if } d = 7, \\
1 - 2^{-(1+2 \gamma)} &\text{otherwise}.
\end{cases}
\end{align*}
It will be convenient to introduce convergence factors in the two Euler products and thus bundle together all singularities in a number of zeta functions and Dirichlet $L$-functions. Doing so, and writing $\chi (p) = (-d/p)$, we find that our expression (\ref{whatamess}) equals
\begin{align*}
&\frac{\zeta(1 + 2 \alpha) L(1+2 \gamma, \chi)}{\zeta(1 + \alpha + \gamma) L(1 + \alpha + \gamma, \chi)} \lp 1 + d^{-(1/2 + \alpha)} \rp \lp 1 - d^{-(1/2 + \gamma)} \rp \lp 1 - d^{-(1+ \alpha + \gamma)} \rp^{-1} H_2(\alpha, \gamma) \\[+1em] 
&\quad \quad \times \prod_{p \geqslant 3 \atop (-d/p) = 1} \frac{\lp 1 - 2p^{-(1 + \alpha + \gamma)} + p^{-(1 + 2 \gamma)} \rp \lp 1 - p^{-(1+2 \gamma)} \rp}{\lp 1 - p^{-(1+\alpha + \gamma)} \rp^2} \\ 
&\quad \quad \times \prod_{p \geqslant 3 \atop (-d/p) = -1} \frac{ \lp 1 - p^{-(1+2 \gamma)} \rp \lp 1 + p^{-(1+2\gamma)} \rp}{\lp 1 - p^{-(1+\alpha + \gamma)} \rp \lp 1 + p^{-(1+\alpha + \gamma)} \rp },
\end{align*}
where the function $H_2$, defined by
\begin{align*}
H_2(\alpha, \gamma) = \begin{cases}
1 &\text{if } d = 2, \\
\lp 1 - 2^{-(1+2 \gamma)} \rp \lp 1 - 2^{-(1 + \alpha + \gamma)} \rp^{-2} \lp 1 - 2^{-(\alpha + \gamma)} + 2^{-(1 + 2 \gamma)} \rp &\text{if } d = 7, \\
\lp 1 - 2^{-2(1+2 \gamma)} \rp \lp 1 - 2^{-2(1 + \alpha + \gamma)} \rp^{-1} &\text{otherwise},
\end{cases}
\end{align*}
keeps track of the contributions from the prime $p = 2$ when $d \neq 2$. Indeed; for example, in case of the $L$-function $L(1 + 2 \gamma, \chi)$, we see that all of its factors are cancelled out by reciprocal factors in the two Euler products, except for factors corresponding to the primes $p = 2$ and $p = d$. If $d = 2$, then no such factor is missing as $\chi(2) = 0$, and if $d \geqslant 3$, only the factor corresponding to $p = 2$ is missing, namely the factor $1 - \chi(2) 2^{-(1 + 2 \gamma)}$. \par 
If we let
\begin{align*}
A_1(\alpha, \gamma) \, \, &:= \prod_{p \geqslant 3 \atop (-d/p) = 1} \frac{\lp 1 - 2p^{-(1 + \alpha + \gamma)} + p^{-(1 + 2 \gamma)} \rp \lp 1 - p^{-(1+2 \gamma)} \rp}{\lp 1 - p^{-(1+\alpha + \gamma)} \rp^2}, \\
A_{-1}(\alpha, \gamma) \, \, &:= \prod_{p \geqslant 3 \atop (-d/p) = -1} \frac{ \lp 1 - p^{-(1+2 \gamma)} \rp \lp 1 + p^{-(1+2\gamma)} \rp}{\lp 1 - p^{-(1+\alpha + \gamma)} \rp \lp 1 + p^{-(1+\alpha + \gamma)} \rp },\\
F_d (\alpha, \gamma)  \, \, &:= \frac{\lp 1 + d^{-(1/2 + \alpha)} \rp \lp 1 - d^{-(1/2 + \gamma) }\rp}{1 - d^{-(1 + \alpha + \gamma)}},
\end{align*}
we therefore have 
\begin{align}\label{partialprogressisalsoprogress}
G(\alpha, \gamma) &= \frac{\zeta(1 + 2 \alpha) L(1+2 \gamma, \chi)}{\zeta(1 + \alpha + \gamma) L(1 + \alpha + \gamma, \chi)} F_d (\alpha, \gamma)  H_2(\alpha, \gamma) A_1(\alpha, \gamma) A_{-1}(\alpha, \gamma).
\end{align}
Combining this with (\ref{Xkaverage}), we see that the prediction of the Ratios Conjecture is that with $\alpha$ and $\gamma$ subject to the conditions (\ref{conjcondsforalphagamma}), the average $R_K (\alpha, \gamma)$ is equal to
\begin{align}\label{RCpred1}
R_K (\alpha, \gamma) &= G(\alpha, \gamma) + \frac{1}{1-2 \alpha}  \lp \frac{2 \pi}{KN \sqrt{|D|}} \rp^{2 \alpha} G(-\alpha, \gamma) + O_{\varepsilon} \lp K^{-1/2 + \varepsilon} \rp.
\end{align}
\subsection{The Logarithmic Derivative}
If we know that the effective estimate (\ref{RCpred1}) is invariant under differentiation (in the sense that the derivatives of the main terms are equal, at least up to an error of roughly the same order of magnitude), we can use (\ref{RCpred1}) to describe an asymptotic for the average of the logarithmic derivative of $L_k$ over the family $\mathcal{F}(K)$. Since the logarithmic derivative $L_k' / L_k$ is intimately connected with the zeros and poles of $L_k$ by Cauchy's residue theorem, such an asymptotic estimate will allow us to study the asymptotics of the one-level density of the family $\mathcal{F}(K)$. \par 
To see that (\ref{RCpred1}) is, in fact, invariant under differentiation in the sense above, let $\Omega \subset \C$ be any open set where the function 
\begin{align*}
\alpha \mapsto f(\alpha) := R_K (\alpha, \gamma)-G(\alpha, \gamma) -  \frac{1}{1-2 \alpha}  \lp \frac{2 \pi}{KN \sqrt{|D|}} \rp^{2 \alpha} G(-\alpha, \gamma)
\end{align*}
is holomorphic. For $\alpha_0 \in \Omega$, choose $\delta > 0$ so that $\Omega$ contains the circle $C$ centered at $\alpha_0$ with radius $\delta$. Then by Cauchy's integral formula,
\begin{align*}
| f'(\alpha_0) | &\leqslant \frac{1}{2 \pi} \int_{C} \left| \frac{f(z)}{\lp z-\alpha_0 \rp^2} \right| \hsp |\dd z| \leqslant \frac{1}{2 \pi} \frac{1}{\delta^2} \cdot 2 \pi \delta \cdot \sup_{z \in C} |f(z)| = \frac{1}{\delta} \cdot O_\varepsilon \lp K^{-1/2 + \varepsilon} \rp,
\end{align*}
and the claim follows by linearity of differentiation. \par 
Since only the numerators in the sum defining $R_K (\alpha, \gamma)$ depend on $\alpha$, (\ref{RCpred1}) and the estimate obtained by differentiating (\ref{RCpred1}) therefore imply that
\begin{align}\label{logdev01}
\begin{split}
\frac{1}{K} \sum_{k = 1}^K \frac{L_k'(1/2 + r)}{L_k (1/2 + r)} &= \frac{\partial}{\partial \alpha} R_K (\alpha, \gamma)\bigg\rvert_{\alpha = \gamma = r} \\
& \approx \frac{\partial}{\partial \alpha} G (\alpha, \gamma) \bigg\rvert_{\alpha = \gamma = r} + \frac{\partial}{\partial \alpha}  \frac{1}{1-2 \alpha}  \lp \frac{2 \pi}{KN \sqrt{|D|}} \rp^{2 \alpha} G(-\alpha, \gamma) \bigg\rvert_{\alpha = \gamma = r}
\end{split}
\end{align}
for any $r \in \C$ satisfying the conditions
\begin{align}\label{condsonr}
\frac{1}{\log K} \ll \mathrm{Re}(r) < 1/4, \quad \quad \mathrm{Im}(r) \ll_\varepsilon K^{1-\varepsilon}.
\end{align}
\par By going through a computation identical to that in the proof of \cite[Lemma 3.4]{waxman} and observing that $F_d(r,r) = H_2 (r,r) = A_1 (r,r) = A_{-1}(r,r) = 1$, we now see that
\begin{align*}
\frac{\partial}{\partial \alpha} G(\alpha, \gamma) \bigg\rvert_{\alpha = \gamma = r} &= \frac{\zeta'(1+2r)}{\zeta(1+2r)} -\frac{L'(1 + 2 r, \chi)}{L(1+2r, \chi)} - \frac{d^{r + 1/2} \log d}{d^{2r+1} - 1} \\
&\quad + H_2'(r) - 2 \sum_{p \geqslant 3 \atop (-d/p) = -1} \frac{\log p}{p^{4r + 2} - 1},
\end{align*}
where 
\begin{align*}
H_2'(r) := \frac{\partial}{\partial \alpha} H_2 (\alpha, \gamma) \bigg\rvert_{\alpha = \gamma = r} = \begin{cases}
0 &\text{if } d = 2, \, 7 \\
-2 \log 2 \lp 2^{2(2r+1)} - 1 \rp^{-1} &\text{otherwise}.
\end{cases}
\end{align*}
\par It remains to compute the other partial derivative in (\ref{logdev01}). However, in the resulting sum only the term coming from differentiating the quotient of zeta- and $L$-functions survives on account of the pole of $\zeta (s)$ at $s = 1$. Thus, with
\begin{align*}
F_d (-r,r) &= 1 + \frac{d^{1/2 + r} - d^{1/2 - r}}{d-1}, \\
H_2 (-r,r) &= \begin{cases}
1 &\text{if } d = 2, \\
2^{1-2r}\lp 1 - 2^{-1-2r} \rp &\text{if } d = 7, \\
\tfrac{4}{3} \lp 1-2^{-2(2r+1)} \rp &\text{otherwise},
\end{cases} \\[+0.5em]
A_1(-r,r) &= \prod_{p \geqslant 3 \atop (-d/p) = 1} \frac{\lp 1 - 2p^{-1} + p^{-(1 + 2 r)} \rp \lp 1 - p^{-(1+2 r)} \rp}{\lp 1 - p^{-1} \rp^2}, \\
A_{-1}(-r,r) &= \prod_{p \geqslant 3 \atop (-d/p) = -1} \frac{ 1 - p^{-2(1+2r)} }{ 1 - p^{-2}},
\end{align*}
we conclude that
\begin{align*}
J(r) :\hspace{-0.3em}&= \frac{\partial}{\partial \alpha}  \frac{1}{1-2 \alpha}  \lp \frac{2 \pi}{KN \sqrt{|D|}} \rp^{2 \alpha} G(-\alpha, \gamma) \bigg\rvert_{\alpha = \gamma = r} \\
&= -\frac{\zeta(1-2r)L(1+2r, \chi)}{L(1, \chi)}  \frac{1}{1-2r} \lp \frac{2 \pi}{KN \sqrt{|D|}} \rp^{2 r} F_d (-r,r) H_2 (-r,r) A_1 (-r,r) A_{-1}(-r,r).
\end{align*}
\par We sum up this discussion in the following proposition. 
\begin{prop}
The Ratios Conjecture implies that for any $r \in \C$ satisfying (\ref{condsonr}), we have
\begin{align*}
\frac{1}{K} \sum_{k = 1}^K \frac{L_k'(1/2 + r)}{L_k (1/2 + r)} &= \frac{\zeta'(1+2r)}{\zeta(1+2r)} -\frac{L'(1 + 2 r, \chi)}{L(1+2r, \chi)} - \frac{d^{r + 1/2} \log d}{d^{2r+1} - 1} + H_2'(r) \\ 
&\quad \quad - 2 \sum_{p \geqslant 3 \atop (-d/p) = -1} \frac{\log p}{p^{4r + 2} - 1} + J(r) + O_\varepsilon \lp K^{-1/2 + \varepsilon} \rp.
\end{align*}
\end{prop}
We now let $f : \R \rightarrow \R$ be an even Schwarz function with $\mathrm{supp} \, \widehat{f}$ compact. By the argument given in the beginning of \cite[Sect. 4]{waxman}, the above result allows us to express the $1$-level density $D \lp \mathcal{F}(K); f \rp$ (conditionally on the Ratios Conjecture) as
\begin{align}\label{oneleveldensity}
\nonumber  &\frac{1}{2 \pi i} \int_{(c)} \frac{1}{K} \lp \sum_{k = 1}^K \lp 2 \cdot \frac{L_k' (1/2 + r)}{L_k (1/2 + r)} - \frac{X_k' (1/2 + r)}{X_k (1/2 + r)} \rp \rp f \lp \frac{i r \log K}{\pi} \rp \hsp \dd r \\
&\quad = S_X + S_\zeta + S_L + S_{A'} + S_J + S_d + S_H + O_\varepsilon \lp K^{-1/2+ \varepsilon} \rp,
\end{align}
where $c$ is any real number satisfying $1 / \log K < c < 1/4$, and 
\begin{align}
\label{sX} S_X &:= -\frac{1}{2 K \log K} \int_{(C)} \sum_{k = 1}^K \frac{X_k' \lp 1/2 + \pi i \tau / \log K \rp }{X_k \lp 1/2 + \pi i \tau / \log K \rp} f(\tau) \hsp \dd \tau, \\[+1em]
\label{szeta} S_\zeta &:= \frac{1}{\log K} \int_{(C)} \frac{\zeta' \lp 1 + 2 \pi i \tau / \log K \rp }{\zeta \lp 1 + 2 \pi i \tau / \log K \rp} f(\tau) \hsp \dd \tau, \\[+1em]
\label{sL} S_L &:= -\frac{1}{\log K} \int_{(C)} \frac{L' \lp 1 + 2 \pi i \tau / \log K, \chi \rp }{L \lp 1 + 2 \pi i \tau / \log K, \chi \rp} f(\tau) \hsp \dd \tau, \\[+1em]
\label{sA'} S_{A'} &:= -\frac{2}{\log K} \int_{(C)} 
\sum_{p \geqslant 3 \atop (-d/p) = -1} \frac{\log p}{p^{4 \pi i \tau / \log K + 2} - 1} f(\tau) \hsp \dd \tau, \\[+0.5em]
\label{sJ} S_J &:= \frac{1}{\log K} \int_{(C)} J\lp \frac{\pi i \tau}{\log K} \rp f(\tau) \hsp \dd \tau, \\[+1em]
\label{sd} S_d &:= -\frac{ \log d}{\log K} \int_{(C)} \frac{d^{\pi i \tau / \log K + 1/2}}{d^{2 \pi i \tau / \log K +1} - 1} f(\tau) \hsp \dd \tau, \\[+1em]
\label{sH} S_H &:= \frac{1}{\log K} \int_{(C)} H_2'\lp \frac{\pi i \tau}{\log K} \rp f(\tau) \hsp \dd \tau,
\end{align}
where we denoted by $(C)$ the set of all $\tau$ with $\mathrm{Im}(\tau) = -c \log K / \pi$. \par 
In order to determine explicitly the prediction of the one-level density $D \lp \mathcal{F}(K); f \rp$ offered by the Ratios Conjecture, our next goal will be to provide estimates of each of these integrals in terms of $f$ and the relevant parameters of our family $\mathcal{F}(K)$. This is the point of the next section.
\section{Computations of the Integrals (\ref{sX})-(\ref{sH})}
In the following we will allow all implicit constants to depend on the test function $f$. We start by recalling and adapting some results from \cite{waxman}. 
\begin{lemma}\label{Szetaeffective} Let $B$ be a positive integer. As $K \rightarrow \infty$, we have
\begin{align*}
S_\zeta = - \frac{f(0)}{2} - \sum_{j = 1}^B \frac{c_j \hat{f}^{(j-1)}(0)}{\lp \log K \rp^j} + O_B \lp \frac{1}{\lp \log K \rp^{B+1}} \rp,
\end{align*}
where $c_1, c_2, c_3, \ldots$ are numbers defined in \cite[eq. (5.4), (5.5)]{waxman}. In particular, 
\begin{align*}
c_1 = 1 + \int_1^\infty \frac{\psi(t) - t}{t^2} \hsp \dd t = - \gamma,
\end{align*}
where $\gamma$ denotes the Euler--Mascheroni constant and $\psi(t) = \sssum_{n \leqslant t} \Lambda(n)$ is the second Chebyshev function.
\end{lemma}
\begin{proof}
The asymptotic expression for $S_\zeta$ is simply \cite[Lemma 5.2]{waxman}, so we only need to prove that
\begin{align}\label{oneminuseulermascheroni}
\int_1^\infty \frac{\psi(t)-t}{t^2} \hsp \dd t = - \gamma - 1.
\end{align}
\par We begin by rewriting the integral as
\begin{align}\label{oneminus--part1}
\int_1^\infty \frac{\psi(t)-t}{t^2} \hsp \dd t = \lim_{T \rightarrow \infty} \int_1^T \frac{\psi(t)-t}{t^2} \hsp \dd t = \lim_{T \rightarrow \infty} \int_1^T \frac{\psi(t)}{t^2} \hsp \dd t - \log T.
\end{align}
To evaluate the integral of $\psi(t) / t^2$, we let $\lbr q_k : k \geqslant 1 \rbr$ denote the sequence of prime powers in increasing order, and we write $p_k := \exp(\Lambda(q_k))$ for the unique prime dividing $q_k$. It will also be convenient to write $\Pi_k := p_1 \cdot p_2 \cdots p_k$. Then, using that $\psi(t)$ is constantly equal to $\log \Pi_k = \psi(q_k)$ on the interval $\lb q_k, q_{k+1} \rp$, we compute that for any large prime power $q_M$,
\begin{align*}
\int_1^{q_{M}} \frac{\psi(t)}{t^2} \hsp \dd t &= \sum_{k=1}^M \log \Pi_k \int_{q_k}^{q_{k+1}} \frac{1}{t^2} \hsp \dd t = \sum_{k=1}^M \log \Pi_k \lp \frac{1}{q_k} - \frac{1}{q_{k+1}} \rp \\
&= \frac{\log \Pi_1}{q_1} - \frac{\log \Pi_M}{q_{M+1}} + \sum_{k = 2}^{M-1} \frac{1}{q_k} \Big( \log \Pi_{k} - \log \Pi_{k-1} \Big) \\
&= - \frac{\log \Pi_M}{q_{M+1}} + \sum_{k = 1}^{M-1} \frac{\log p_k}{q_k} = - \frac{\psi(q_M)}{q_{M+1}} + \sum_{n = 1}^{q_{M-1}} \frac{\Lambda(n)}{n}.
\end{align*}
Inserting this into \eqref{oneminus--part1} and replacing $T$ by $q_{M-1}$, we obtain from the prime number theorem that
\begin{align*}
\int_1^\infty \frac{\psi(t)-t}{t^2} \hsp \dd t &= \lim_{M \rightarrow \infty} - \frac{\psi(q_M)}{q_{M}} \cdot \frac{q_M}{q_{M+1}} + \sum_{n = 1}^{q_{M-1}} \frac{\Lambda(n)}{n} - \log q_{M-1} + \log \frac{q_{M-1}}{q_M} \\
&= -1 + \gamma + \sum_{n = 1}^\infty \frac{\Lambda (n) - 1}{n} = \lim_{s \rightarrow 1^+} \lp -\frac{\zeta'(s)}{\zeta(s)} - \zeta(s) \rp = -1 - \gamma,
\end{align*}
where we also used the familiar asymptotic for the harmonic numbers and noted that $q_M / q_{M+1} \rightarrow 1$ as $M \rightarrow \infty$. 
\end{proof}
Although our Dirichlet $L$-function $L(s , \chi)$ is defined in terms of the character $\chi ( \cdot ) = \lp -d/ \cdot \rp$ and not the non-principal character modulo $4$ as in \cite[Lemma 5.4]{waxman}, we note that this result, including the statement about $S_{A'}$, remains valid in our case with exactly the same proof. We record these two results in the following two separate lemmas. 
\begin{lemma}[{\cite[Lemma 5.4]{waxman}}] As $K \rightarrow \infty$, we have 
\begin{align*}
S_L = -\frac{\hat{f}(0)}{\log K} \frac{L'(1, \chi)}{L(1, \chi)} + O \lp \frac{1}{(\log K )^2} \rp.
\end{align*}
\end{lemma}
\begin{lemma}[{\cite[Lemma 5.4]{waxman}}] As $K \rightarrow \infty$, we have 
\begin{align*}
S_{A'} = -\frac{2 \hat{f}(0)}{\log K} \sum_{p \geqslant 3 \atop (-d/p) = -1} \frac{\log p}{p^2 - 1} + O \lp \frac{1}{(\log K )^2} \rp.
\end{align*}
\end{lemma}
Moreover, we can adapt \cite[Lemma 5.1]{waxman} to our situation and obtain the following result. 
\begin{lemma} As $K \rightarrow \infty$, we have
\begin{align*}
S_X = \hat{f}(0) \lp 1 + \frac{\log \sqrt{|D|} - \log 2 \pi + \log N - 1}{\log K} \rp + O \lp \frac{1}{K} \rp.
\end{align*}
\end{lemma}
\begin{proof}
We have 
\begin{align*}
\frac{X_k'(s)}{X_k(s)} = \frac{\dd}{\dd s} \log \Gamma(1 - s+ Nk) - \frac{\dd}{\dd s} \log \Gamma (s + Nk) - \log |D| + 2 \log 2 \pi.
\end{align*}
The contribution of the constants to the integral (\ref{sX}) is
\begin{align*}
\frac{1}{\log K} \widehat{f}(0) \lp \log \sqrt{|D|} - \log 2 \pi \rp.
\end{align*} 
As for the contribution of the gamma functions, we recall that, in the notation of \cite{waxman}, we have $\textbf{k} = 1/2 + Nk$. It then follows immediately from \cite[Eq. (5.1)]{waxman}, that this contribution equals
\begin{align*}
\frac{1}{K \log K} \widehat{f}(0) \sum_{k = 1}^K \log \textbf{k} + O \lp \frac{1}{K} \rp.
\end{align*}
By arguing as in \cite[Eq. (5.2)]{waxman}, we find that
\begin{align*}
\sum_{k = 1}^K \log \, \textbf{k} &= \sum_{k = 1}^K \log \lp 1/2 + Nk \rp = K \log N + \log K! + O ( \log K ) \\
&= K \log N + K \log K - K + O(\log K) \\[+0.6em]
&= K \log K + K \lp \log N - 1 \rp + O(\log K).
\end{align*}
This proves the claim.
\end{proof}
We now turn to the integral (\ref{sJ}). As Lemma \ref{SJ-est} below will show, the asymptotic expression for this integral will involve a special value of the logarithmic derivative of $L(s, \chi)$. In anticipation of this, we will first compute this special value.
\begin{lemma}\label{kroneckerlimitformulastuff}
Let $\eta$ denote Dedekind's eta function,
\begin{align*}
\eta (\tau) = e^{\pi i \tau / 12} \prod_{n \geqslant 1} \lp 1 - e^{2 \pi i n \tau} \rp, \quad \quad \im(\tau) > 0,
\end{align*}
and let $\gamma$ denote the Euler--Mascheroni constant. Then we have 
\begin{align*}
\frac{L'(1, \chi)}{L(1, \chi)} = \gamma - \log 2 - \lp \log |D| \rp / 2 - \log \im(\tau_0) - 4 \log \left| \eta (\tau_0) \right|
\end{align*}
where $C(d) = 2 \gamma - \log 2 - \lp \log |D| \rp / 2 - \log \im(\tau_0) - 4 \log \left| \eta (\tau_0) \right|$, and 
\begin{align*}
\tau_0 = \begin{cases}
i\sqrt{2} &\text{ if } d = 2,\\
(-1 + i \sqrt{d})/2 &\text{ if } d \neq 2.
\end{cases}
\end{align*}
\end{lemma}
\begin{proof}
If $\zeta_\K (s)$ denotes the Dedekind zeta function of $\K$, \cite[Prop. 10.5.5]{cohen} gives the factorization $\zeta_\K (s) = \zeta(s) L(s, \chi)$. Indeed, this is clear if $d \geqslant 3$ since then $D = -d$, whereas for $d = 2$ and $n \geqslant 1$, the fact that $(-2/n)$ can only assume the values $0, \pm 1$ immediately shows that
\begin{align*}
\lp \frac{-2}{n} \rp = \lp \frac{-2}{n} \rp^3 = \lp \frac{-8}{n} \rp,
\end{align*}
so that $(-d/n) = (D/n)$ even in this case. \par On the other hand, since any non-zero ideal $\mathfrak{m} \subset \ringint_\K$ has exactly $\left| \ringint_\K^\times \right|$ generators, we also have
\begin{align*}
\zeta_\K (s) = \frac{1}{\left| \ringint_\K^\times \right|} \sum_{N(a,b) \neq 0} \frac{1}{N(a,b)^s}, \quad \quad \re(s) > 1,
\end{align*}
where $N(a,b)$ denotes the norm of the element $j(a,b)$ as defined in \textsc{Section 2.2}. It now follows from this and \cite[Corollary 10.4.8]{cohen} that
\begin{align}\label{zetaL-laurent}
\zeta(s) L( \chi, s) &= \frac{1}{\left| \ringint_\K^\times \right|} \sum_{N(a,b) \neq 0} \frac{1}{N(a,b)^s} = \frac{2}{\left| \ringint_\K^\times \right|} \frac{\pi}{\sqrt{|D|}} \lp \frac{1}{s-1} + C(d) + O(s-1) \rp.
\end{align}
Furthermore, by writing $L$ and $\zeta$ as Laurent series around $s = 1$, we find that
\begin{align*}
L(s, \chi) &= L(1, \chi) + L'(1, \chi) (s-1) + O\lp (s-1)^2 \rp, \\
\zeta(s) &= \frac{1}{s-1} + \gamma + O(s-1),
\end{align*}
and hence 
\begin{align}\label{zetaL-laurent2}
\zeta(s) L(s, \chi)  = \frac{L(1, \chi)}{s-1} + L'(1, \chi) + \gamma L(1, \chi) + O(s-1).
\end{align}
Together, the two different expressions (\ref{zetaL-laurent}) and (\ref{zetaL-laurent2}) for $\zeta(s) L(s, \chi)$ now force an equality of coefficients, namely
\begin{align*}
L(1, \chi) = \frac{2 \pi}{\left| \ringint_\K^\times \right| \sqrt{|D|}}, \quad \quad L'(1,\chi) = \frac{2 \pi \lp C(d) - \gamma \rp}{\left| \ringint_\K^\times \right| \sqrt{|D|}}.
\end{align*}
By computing the quotient $L'(1 , \chi) / L(1, \chi)$, we obtain the claim.
\end{proof} 
\begin{lemma}\label{SJ-est}
As $K \rightarrow \infty$, we have
\begin{align*}
S_J = \frac{f(0)}{2} - \frac{1}{2} \int_{\R} \hat{f}(\tau) \mathbbm{1}_{\lb -1, 1 \rb}(\tau) \hsp \dd \tau + \frac{ \hat{f}(1) }{\log K}\ell_1 + O_d \lp \frac{1}{(\log K)^2} \rp,
\end{align*}
where
\begin{align*}
\ell_1 &= \frac{L'(1, \chi)}{L(1,\chi)} + 2 \sum_{p \geqslant 3 \atop (-d/p) = -1} \frac{\log p}{p^2 - 1} +  \log \frac{2\pi e}{N\sqrt{|D|}} + \frac{\sqrt{d} \log d}{d-1} - \gamma - \frac{2a \log 2}{3},
\end{align*}
$a = -\mathbbm{1}(d \neq 2,7)$, and $L'(1, \chi) / L(1, \chi)$ is given in Lemma \ref{kroneckerlimitformulastuff}.
\end{lemma}
\begin{proof}
We proceed as in the proof of \cite[Lemma 5.5]{waxman}, which relies on the methods of \cite{FPS3}. That is, we will replace the domain of integration $(C)$ with the union of the compact and non-compact contours $C_1 \cup C_\eta$ and $C_0$, respectively, where
\begin{align*}
C_0 &:= \lbr \tau \in \C : \mathrm{Im}(\tau) = 0, \, |\mathrm{Re}(\tau)| \geqslant (\log K)^\varepsilon \rbr, \\[+0.4em]
C_1 &:= \lbr \tau \in \C : \mathrm{Im}(\tau) = 0, \, \eta \leqslant |\mathrm{Re}(\tau) | \leqslant (\log K)^\varepsilon \rbr, \\
C_\eta &:= \lbr \eta e^{i \theta} : - \pi \leqslant \theta \leqslant 0 \rbr.
\end{align*}
The utility of such a decomposition is two-fold: On the non-compact part $C_0$ we can bound the integrand using the rapid decay of our test function $f$. On the other hand, on the compact part we can estimate the individual factors in the integrand with the first few terms of their Taylor expansions. Moreover, the fact that $C_1 \cup C_\eta$ becomes symmetric when we let $\eta \rightarrow 0$ means that we do not have to take into account contributions from any odd, positive powers of $\tau$ in these expansions, as the integrals of $\tau f(\tau)$, $\tau^3 f(\tau)$, $\tau^5 f(\tau)$, $\ldots$ over this set vanish. We exploit this fact to get rid of any occurrences of $\varepsilon$ in the error terms. We now proceed to the details. \par 
Initially, we recall the following elementary estimates: If $\gamma$ denotes the Euler--Mascheroni constant, then as $K \rightarrow \infty$,
\begin{align}
\label{SJ-zeta} \zeta \lp 1 - \frac{2 \pi i \tau}{\log K} \rp &= - \frac{\log K}{2 \pi i \tau} + \gamma + O \lp \frac{|\tau|}{\log K} \rp , \\[+0.6em]
\label{SJ-L} \frac{L(1 + 2 \pi i \tau / \log K , \chi)}{L(1, \chi)} &= 1 + \frac{L'(1, \chi)}{L(1,\chi)} \frac{2 \pi i \tau}{\log K} + O \lp \frac{|\tau|^2}{( \log K )^2} \rp,
\end{align}
and
\begin{align} \label{SJ-A}
\begin{split}
&A_1 \lp - \frac{\pi i \tau}{\log K},  \frac{\pi i \tau}{\log K} \rp A_{-1} \lp - \frac{\pi i \tau}{\log K},  \frac{\pi i \tau}{\log K} \rp  \\
&\quad = 1 + \lp 2 \sum_{p \geqslant 3 \atop (-d/p) = -1} \frac{\log p}{p^2 - 1} \rp \frac{2 \pi i \tau}{\log K} + O \lp \frac{|\tau|^2}{( \log K )^2} \rp.
\end{split}
\end{align}
\par We now obtain Taylor expansions of the other factors in $J$. First of all,
\begin{align}\label{SJ-frac-taylor}
\begin{split}
&\frac{1}{1-2 r} \lp \frac{2 \pi}{KN \sqrt{|D|}} \rp^{2r} \bigg\rvert_{r = \tfrac{\pi i \tau}{\log K}} \\
&\quad = \exp \lp \frac{2 \pi i \tau}{\log K} \log \frac{2\pi}{N\sqrt{|D|}} - 2 \pi i \tau \rp \lp 1 + \frac{2 \pi i \tau}{\log K} + O \lp \frac{|\tau|^2}{( \log K )^2} \rp \rp \\
&\quad = e^{- 2 \pi i \tau} \lp 1 + \frac{2 \pi i \tau}{\log K} \log \frac{2\pi}{N\sqrt{|D|}}
+ O \lp \frac{|\tau|^2 (\log |D|)^2}{(\log K)^2} \rp \rp \lp 1 + \frac{2 \pi i \tau}{\log K} + O \lp \frac{|\tau|^2}{( \log K )^2} \rp \rp \\
&\quad = e^{- 2 \pi i \tau} + e^{-2 \pi i \tau} \frac{2 \pi i \tau}{\log K} \log \frac{2 \pi e}{N \sqrt{|D|}} + O \lp \frac{|\tau|^2 (\log |D|)^2}{(\log K)^2} \rp,
\end{split}
\end{align}
whenever $K$ is so large that 
\begin{align*}
\left| \frac{2 \pi i \tau}{\log K} \right| < 1, \quad \quad \left| \frac{2 \pi i \tau}{\log K} \log \frac{2 \pi}{N \sqrt{|D|}} \right| < 1.
\end{align*}
(Note that this is certainly satisfied for $\tau \in C_1 \cup C_\eta$.) Next, we see that
\begin{align}\label{SJ-Fdbound}
F_d \lp - \frac{\pi i \tau}{\log K} , \frac{\pi i \tau}{\log K} \rp = 1 + \frac{\sqrt{d} \log d}{d-1} \frac{2 \pi i \tau}{\log K} + O_d \lp \frac{|\tau|^2}{(\log K)^2} \rp.
\end{align}
Finally, we also record the bound
\begin{align}\label{SJ-Hbound}
H_2 \lp - \frac{\pi i \tau}{\log K} , \frac{\pi i \tau}{\log K} \rp = 1 - \frac{2a \log 2}{3} \frac{2 \pi i \tau}{\log K} + O \lp \frac{|\tau|^2}{(\log K)^2} \rp.
\end{align}
\par By taking the product of all the Taylor expansions (\ref{SJ-zeta})-(\ref{SJ-Hbound}) and disregarding all those resulting terms which have order at least $\tau / \log K$, we obtain that
\begin{align*}
&J \lp \frac{\pi i \tau}{\log K} \rp \\[+1em]
&\quad = \lp \frac{1}{x} - \gamma + O\lp | x | \rp \rp \Biggl( e^{-2 \pi i \tau} + x \Biggl( e^{-2 \pi i \tau} \frac{L'(1, \chi)}{L(1,\chi)} {\color{white} \Biggr)} + 2 e^{-2 \pi i \tau} \sum_{p \geqslant 3 \atop (-d/p) = -1} \frac{\log p}{p^2 - 1} \\ 
&\quad \quad \hspace{0.8em} + \, e^{-2 \pi i \tau}  \log \frac{2\pi e}{N\sqrt{|D|}} \Biggl. \Biggl. +  \, e^{-2 \pi i \tau} \frac{\sqrt{d} \log d}{d-1} - e^{-2 \pi i \tau} \frac{2a \log 2}{3} \Biggr) + O_d \lp |x|^2 \rp \Biggr) \\[+0.5em]
&\quad = \frac{e^{-2 \pi i \tau}}{x} + \, e^{-2 \pi i \tau} \Biggl( - \gamma + \frac{L'(1, \chi)}{L(1,\chi)} + 2 \sum_{p \geqslant 3 \atop (-d/p) = -1} \frac{\log p}{p^2 - 1} {\color{white} \Biggr)} \\
&\quad \quad  {\color{white} \Biggl)} + \log \frac{2\pi e}{N\sqrt{|D|}} + \frac{\sqrt{d}\log d}{d-1} - \frac{2a \log 2}{3} \Biggr) + O_d \lp |x | \rp,
\end{align*}
where we wrote $x = 2 \pi i \tau / \log K$ for simplicity. Thus, we obtain
\begin{align}\label{SJ-whatamess}
\begin{split}
S_J &= \int_{C_1 \cup C_\eta} f(\tau) \frac{e^{-2\pi i \tau}}{2 \pi i \tau} \hsp \dd \tau \\[+0.8em]
&\quad + \frac{1}{\log K} \Biggl( - \gamma + \frac{L'(1, \chi)}{L(1,\chi)} + 2 \sum_{p \geqslant 3 \atop (-d/p) = -1} \frac{\log p}{p^2 - 1} + \log \frac{2\pi e }{N\sqrt{|D|}} + \frac{\sqrt{d}\log d}{d-1} - \frac{2a \log 2}{3} \Biggr) \\
&\quad \times \int_{C_1 \cup C_\eta} f (\tau) e^{-2 \pi i \tau} \hsp \dd \tau + \frac{1}{\log K} \int_{C_0} J\lp \frac{\pi i \tau}{\log K} \rp f(\tau) \hsp \dd \tau + O_d \lp \frac{1}{(\log K)^{2}} \rp,
\end{split}
\end{align}
where we made use of the evenness of $f$ as described earlier. By arguing exactly as in \cite[Lemma 5.5]{waxman}, we relate the integrals above to special values of the Fourier transform $\hat{f}$, namely
\begin{align}
\label{SJ-fhat1} \int_{C_1 \cup C_\eta} f(\tau) e^{-2\pi i \tau} \hsp \dd \tau &= \hat{f}(1) + O \lp \frac{1}{(\log K)^3} \rp, \\
\label{SJ-fhat0} \int_{C_1 \cup C_\eta} f(\tau) \frac{e^{- 2 \pi i \tau}}{2 \pi i \tau} \hsp \dd \tau &= \frac{f(0)}{2} - \frac{1}{2} \int_{-1}^1 \hat{f}(\tau) \hsp \dd \tau + O \lp \frac{1}{(\log K)^3} \rp.
\end{align}
Since the rapid decay of $f$ on $\R$ shows that the integral over $C_0$ in (\ref{SJ-whatamess}) is at most a constant times $(\log K)^{-2}$ (for example), it now follows from (\ref{SJ-fhat1}) and (\ref{SJ-fhat0}) that
\begin{align*}
S_J &= \frac{f(0)}{2} - \frac{1}{2} \int_{-1}^1 \hat{f}(\tau) \hsp \dd \tau + \frac{\hat{f}(1)}{\log K} \Biggl( \frac{L'(1, \chi)}{L(1,\chi)} + 2 \sum_{p \geqslant 3 \atop (-d/p) = -1} \frac{\log p}{p^2 - 1} {\color{white} \Biggr)} \\
&\quad \quad{\color{white} \Biggl)} + \log \frac{2\pi e}{N\sqrt{|D|}}  + \frac{\sqrt{d} \log d}{d-1} - \gamma - \frac{2a \log 2}{3} \Biggr) + O_d \lp \frac{1}{(\log K)^{2}} \rp,
\end{align*}
which completes the proof.
\end{proof}
We now turn to the final two integrals (\ref{sd}) and (\ref{sH}).
\begin{lemma}
As $K \rightarrow \infty$, we have the estimate
\begin{align*}
S_d = -\frac{ \log d}{\log K} \frac{\sqrt{d}}{d-1} \hat{f}(0) + O_d \lp \frac{1}{( \log K )^{2}} \rp.
\end{align*}
\end{lemma}
\begin{proof}
The integrand only has poles when $\mathrm{Im}(\tau) = (\log K) / 2 \pi > 0$, so analogously to the proof of Lemma \ref{SJ-est}, we use Cauchy's residue theorem and the rapid decay of $f$ to move the contour $(C)$ to the real line without changing the value of the integral. As before, we partition this set into a compact and a non-compact part in order to, respectively, use the (even-indexed terms of the) Taylor expansion of the integrand and bound the integral using the decay of the test function. Concretely, we write $\R = C_0 \cup C_1$ with
\begin{align*}
C_0 &:= \lbr \tau \in \C : \mathrm{Im}(\tau) = 0, \, |\mathrm{Re}(\tau)| > (\log K)^\varepsilon \rbr, \\[+0.4em]
C_1 &:= \lbr \tau \in \C : \mathrm{Im}(\tau) = 0, \, |\mathrm{Re}(\tau) | \leqslant (\log K)^\varepsilon \rbr.
\end{align*}
\par As in the proof of Lemma \ref{SJ-est}, we note that
\begin{align*}
\int_{C_0} \frac{d^{\pi i \tau / \log K + 1/2}}{d^{2 \pi i \tau / \log K +1} - 1} f(\tau) \hsp \dd \tau \ll_d \frac{1}{(\log K)^2}, 
\end{align*}
whereas for the integral over $C_1$, we use the Taylor expansion
\begin{align*}
\frac{d^{r + 1/2}}{d^{2r + 1} - 1} = \frac{\sqrt{d}}{d-1} + r \cdot \frac{\dd}{\dd r} \lp \frac{d^{r + 1/2}}{d^{2r + 1} - 1} \rp \Biggr\vert_{r = 0} + O_d \lp r^2 \rp \quad \quad (r \rightarrow 0)
\end{align*}
and the fact that $\tau f(\tau)$ is odd to obtain
\begin{align*}
\int_{C_1} \frac{d^{\pi i \tau / \log K + 1/2}}{d^{2 \pi i \tau / \log K +1} - 1} f(\tau) \hsp \dd \tau &= \int_{C_1} \lp \frac{\sqrt{d}}{d-1} + O_d \lp \frac{\tau^2}{\lp \log K \rp^2} \rp \rp f(\tau) \hsp \dd \tau \\
&= \lp \frac{\sqrt{d}}{d-1} + O_d \lp \frac{1}{\lp \log K \rp^{2-2\varepsilon}}\rp \rp \int_{C_1} f(\tau) \hsp \dd \tau \\
&= \lp \frac{\sqrt{d}}{d-1} + O_d \lp \frac{1}{\lp \log K \rp^{2-2\varepsilon}} \rp \rp \lp \int_{\R} f(\tau) \hsp \dd \tau + O \lp \frac{1}{(\log K)^2}\rp \rp,
\end{align*}
where we used the rapid decay of $f$ in the last step. The claim now follows.
\end{proof}
Finally, we have the following asymptotic estimate. We assume that $d \neq 2, 7$ since otherwise $S_H = 0$.
\begin{lemma}\label{SHeffective}
Suppose that $d \neq 2, \, 7$. As $K \rightarrow \infty$, we have
\begin{align*}
S_H = \frac{-2 \log 2}{3 \log K} \hat{f}(0) + O \lp \frac{1}{( \log K )^{2}} \rp.
\end{align*}
\end{lemma}
\begin{proof}
The method of proof is identical to that of the previous lemma. Once again, we note that $H_2'(\pi i \tau / \log K)$ only has poles if $\mathrm{Im}(\tau) = \log K / (2 \pi \log 2) > 0$, so that we are justified in moving the contour to the real line. As before, the integral over the non-compact part $C_0$ of our partition $C_0 \cup C_1$ of $\R$ simply contributes to the error term, while the integral over the compact part is 
\begin{align*}
\int_{C_1} \lp \frac{-2 \log 2}{3} + \frac{\pi i \tau}{\log K} \cdot \frac{\dd}{\dd r} H_2'(r) \Bigg\rvert_{r = 0} + O\lp \frac{|\tau|^2}{\lp \log K \rp^2} \rp \rp f(\tau) \hsp \dd \tau,
\end{align*}
which follows from the Taylor expansion
\begin{align*}
H_2 ' (r) = \frac{-2 \log 2}{3} + r \cdot \frac{\dd}{\dd r} H_2'(r) \Bigg\rvert_{r = 0} + O\lp r^2 \rp.
\end{align*}
We now proceed exactly as in the proof of the previous lemma.
\end{proof}
By combining the results from Lemma \ref{Szetaeffective} to Lemma \ref{SHeffective} with (\ref{oneleveldensity}), we have completed the proof of Theorem \ref{OneLevelDensityThm-RC}.
\section{An Unconditional Asymptotic for the One-Level Density}
In this section, we use the following formula for logarithmic derivatives of the $L$-functions $L_k$ in order to give an unconditional expression for the one-level density $D \lp \mathcal{F}(K);f \rp$.
\begin{lemma}\label{logderofL} For $k \geqslant 1$ and $\re (s) > 1$, we have 
\begin{align*}
\frac{L_k'(s)}{L_k(s)} = - \sum_{n \geqslant 1}\frac{c_k(n)}{n^s}, \quad \quad c_k(n) = \Lambda(n) \sum_{\norm (\mathfrak{p}^m) = n} \Big( 1 +  \mathbbm{1} \big( \mathfrak{p} = \langle p \rangle \big) \Big) \psi_k \lp \mathfrak{p}^m \rp,
\end{align*}
where we understand that the indicator function specifies whether or not $\mathfrak{p}$ lies over an inert rational prime.
\end{lemma}
\begin{proof}
By taking the logarithmic derivative of the (absolutely convergent) Euler product that defines $L_k(s)$ in the half-plane $\re(s) > 1$, we get 
\begin{align*}
\frac{L_k'(s)}{L_k(s)} 
= - \sum_{\mathfrak{p}} \log \norm(\mathfrak{p}) \frac{\psi_k(\mathfrak{p})/\norm(\mathfrak{p})^s}{1 - \psi_k(\mathfrak{p})/\norm(\mathfrak{p})^s} 
= - \sum_{\mathfrak{p}} \log \norm(\mathfrak{p}) \sum_{m \geqslant 1} \frac{\psi_k \lp \mathfrak{p}^m \rp}{\norm\lp \mathfrak{p}^m \rp^s},
\end{align*}
where we also used that $\psi_k$ and the norm are completely multiplicative. \par 
To finish the computation we note that, if $\mathfrak{p}$ lies over a ramified or split prime $p$, then for any $m \geqslant 1$, 
\begin{align*}
\log \norm(\mathfrak{p}) = \log p = \Lambda \lp \norm ( \mathfrak{p} ) \rp = \Lambda \lp \norm ( \mathfrak{p}^m ) \rp ,
\end{align*}
whereas if $\mathfrak{p} = \langle p \rangle$ lies over an inert prime, 
\begin{align*}
\log \norm(\mathfrak{p}) = 2 \log p = 2 \Lambda \lp \norm ( \mathfrak{p} ) \rp = 2 \Lambda \lp \norm ( \mathfrak{p}^m ) \rp 
\end{align*}
for any $m \geqslant 1$. In light of this and the previous computation, we can rewrite the logarithmic derivative as
\begin{align*}
\frac{L_k'(s)}{L_k(s)} 
&= - \sum_{\mathfrak{p}} \Big( 1 +  \mathbbm{1} \big( \mathfrak{p} = \langle p \rangle \big) \Big)  \sum_{m \geqslant 1} \frac{\psi_k \lp \mathfrak{p}^m \rp}{\norm \lp \mathfrak{p}^m \rp^s} \Lambda \lp \norm \lp \mathfrak{p}^m \rp \rp.
\end{align*}
We now obtain the claim by grouping together all $\mathfrak{p}^m$ with norm $n$, for all $n \geqslant 1$. 
\end{proof}
Towards a computation of the $1$-level density, we note that just as in \cite[Section 6]{waxman}, we have
\begin{align*}
D(\mathcal{F}(K);f) = \frac{1}{2 \pi i } \int_{(c')} \frac{1}{K} \lp \sum_{k = 1}^K 2\frac{L_k'(1/2 + r)}{L_k (1/2 + r)} - \frac{X_k'(1/2 + r)}{X_k (1/2 + r)} \rp f \lp \frac{i r \log K}{\pi} \rp \hsp \dd r,
\end{align*}
where now $c' \geqslant 1/2$. Using Lemma \ref{logderofL} and arguing as in \cite[Thm. 5.12]{iwakow}, we now obtain
\begin{align*}
D \lp \mathcal{F}(K); f \rp &= S_X - \frac{1}{\pi i} \frac{1}{K} \int_{(c')}  \, \sum_{k = 1}^K \, \sum_{n \geqslant 1}\frac{c_k(n)}{n^{1/2 + r}} f \lp \frac{ir \log K}{\pi} \rp \hsp \dd r \\
&= S_X - \frac{1}{\pi i} \frac{1}{K} \sum_{k = 1}^K \, \sum_{n \geqslant 1} \frac{c_k(n)}{\sqrt{n}} \int_{(c')}   e^{-r \log n} f \lp \frac{ir \log K}{\pi} \rp \hsp \dd r \\
&= S_X - \frac{1}{K \log K} \sum_{k = 1}^K \, \sum_{n \geqslant 1} \frac{c_k(n)}{\sqrt{n}} \int_{\lbr \im(r) = c' \log K/\pi \rbr} e^{2 \pi i r  \log n / 2 \log K} f (r) \hsp \dd r \\
&= S_X - \frac{1}{K \log K} \sum_{k = 1}^K \sum_{n \geqslant 1} \frac{c_k(n)}{\sqrt{n}} \widehat{f} \lp \frac{\log n}{2 \log K} \rp,
\end{align*}
where we replaced $\lbr \im(r) = c' \log K/\pi \rbr$ with $\R$ due to the rapid decay of $f$ on horizontal strips. 
\par The occurrence of the von Mangoldt function in $c_k(n)$ means that the only indices $n$ contributing to the sum above are the prime powers. If we replace $n$ with $p^n$ (to save notation), we can then describe the resulting values of $c_k(p^n)$ depending on the splitting behavior of $p$. Thus, with the help of Lemma \ref{standardalgnt}, we find that
\begin{align*}
c_k (p^n) = 
\begin{cases}
\lp \psi_k(\mathfrak{q}_1^n) + \psi_k(\mathfrak{q}_2^n) \rp \log p &\text{if } \langle p \rangle = \mathfrak{q}_1 \mathfrak{q}_2, \, \mathfrak{q}_1  \neq \mathfrak{q}_2, \\
2 \log p &\text{if } p \text{ is inert and } n \text{ is even}, \\
0 &\text{if } p \text{ is inert and } n \text{ is odd}, \\
\log p &\text{if } \langle p \rangle = \mathfrak{q}^2.
\end{cases}
\end{align*}
Considering these special values of $c_k(p^n)$, we now define 
\begin{align*}
S_\mathrm{inert} &= -\frac{2}{\log K} \sum_{\text{$p$ inert}} \hsp \sum_{n \geqslant 1} \frac{\log p}{p^n} \widehat{f} \lp \frac{n \log p}{\log K} \rp, \\
S_\mathrm{split} &= - \frac{1}{K \log K}  \sum_{\text{$p$ split}} \hsp \sum_{n \geqslant 1} \frac{\log p}{p^{n/2}} \widehat{f} \lp \frac{n \log p}{2 \log K} \rp \sum_{k = 1}^K \lp \psi_k(\mathfrak{q}_1^n) + \psi_k(\mathfrak{q}_2^n) \rp, \\[+0.6em]
S_\mathrm{ram} &= - \frac{1}{\log K} \sum_{n \geqslant 1} \frac{\log d}{d^{n/2}} \widehat{f} \lp \frac{n \log d}{2 \log K} \rp.
\end{align*}
Thus,
\begin{align}\label{unconditionaloneleveldensity}
D \lp \mathcal{F}(K); f \rp = S_X + S_\mathrm{inert} + S_\mathrm{split} + S_\mathrm{ram}.
\end{align}
\par We now want to compare the terms $S_\mathrm{inert}$ and $S_\mathrm{ram}$ with the terms (\ref{sX})-(\ref{sH}) appearing in the expression of the one-level density conditional on the Ratios Conjecture. To facilitate this comparison, we begin by rewriting these terms as follows.
\begin{lemma}\label{lemma4.2equiv}
We have that 
\begin{align}
\label{rewr-zeta} S_\zeta &= -\frac{1}{\log K} \sum_{n \geqslant 1} \frac{\Lambda (n)}{n} \widehat{f}\lp \frac{\log n}{\log K} \rp, \\[+0.9em] 
\label{rewr-A'} S_{A'} &= -\frac{2}{\log K} \sum_{p \geqslant 3 \atop (-d/p)=-1} \sum_{n \geqslant 1} \frac{\log p}{p^{2n}} \widehat{f} \lp \frac{2n \log p}{\log K} \rp, \\[+0.5em]
\label{rewr-L} S_L &= \frac{1}{\log K} \sum_{n \geqslant 1} \frac{\Lambda (n) \chi(n)}{n} \widehat{f} \lp \frac{\log n}{\log K} \rp, \\[+0.9em]
\label{rewr-d} S_d &= - \frac{\log d}{\sqrt{d} \log K} \sum_{n \geqslant 0} \frac{1}{d^n} \widehat{f} \lp \frac{(1/2 + n) \log d}{\log K} \rp,  \\[+0.9em]
\label{rewr-H} S_H &= \frac{a}{2} \frac{\log 2}{\log K} \sum_{n \geqslant 0} \frac{1}{4^n} \widehat{f} \lp \frac{(2n+2) \log 2}{\log K} \rp.
\end{align}
\end{lemma}
\begin{proof}
The proofs of the equalities (\ref{rewr-zeta}) and (\ref{rewr-A'}) follow immediately from the corresponding proofs in \cite[Lemma 4.2]{waxman}, once we substitute our character $\chi$ for the character $\chi_1$ in that paper. \par 
As for the equality (\ref{rewr-L}), we note the standard formula
\begin{align*}
\frac{L'(s,\chi)}{L(s, \chi)} = -\sum_{n \geqslant 1} \frac{\Lambda (n) \chi(n)}{n^s}, \quad \quad \re (s) > 1.
\end{align*}
Substituting this infinite sum for the logarithmic derivative appearing in the definition of $S_L$, we obtain the claim by changing variables and moving the contour of integration to the real line, which is justified due to the rapid decay of $f$ in combination with the lack of poles of the integrand in the part of the complex plane enclosed by these contours. \par 
Turning to $S_d$, we note that when $\tau$ has imaginary part $-c \log K / \pi$ where $c > 0$, the number $-2 \pi i \tau / \log K - 1$ has real part $-2c - 1 < 0$, and so we can write
\begin{align*}
\frac{d^{\pi i \tau / \log K + 1/2}}{d^{2 \pi i \tau / \log K + 1} - 1} &= d^{-\pi i \tau / \log K - 1/2} \sum_{n \geqslant 0} \lp d^{-2 \pi i \tau / \log K - 1} \rp^{n}.
\end{align*}
Inserting this into the expression (\ref{sd}), we get that
\begin{align*}
S_d &= -\frac{ \log d}{\sqrt{d} \log K} \int_{(C)} \exp \lp - \pi i \tau \frac{\log d}{\log K} \rp \sum_{n \geqslant 0} \exp \lp -n \log d \lp \frac{2 \pi i \tau}{\log K} + 1 \rp \rp f(\tau) \hsp \dd \tau \\
&= - \frac{\log d}{\sqrt{d} \log K} \sum_{n \geqslant 0} d^{-n} \int_{(C)} f(\tau) \exp \lp - 2 \pi i \tau (1/2 + n ) \frac{\log d}{\log K} \rp \hsp \dd \tau,
\end{align*}
which equals the claimed expression for $S_d$ as we can move the contour $(C)$ to the real line for the usual reasons.\par 
Finally, we turn our attention to $S_H$.  In the same way as before, we rewrite the function
\begin{align*}
H_2'(r) = \begin{cases}
0 &\text{if } d = 2, 7, \\
-2 \log 2 \lp 2^{2(2r+1)} - 1 \rp^{-1} &\text{otherwise},
\end{cases}
\end{align*}
as a geometric series. Letting $r = \pi i \tau / \log K$, we then find that
\begin{align*}
H_2' \lp \frac{\pi i \tau}{\log K} \rp = \frac{a \log 2}{2} \sum_{n \geqslant 0} 4^{-n}  \exp \lp - 2 \pi i \tau \frac{(2n+2) \log 2}{\log K} \rp,
\end{align*}
where $a = - \mathbbm{1}(d \neq 2,7)$. The equality (\ref{rewr-H}) now follows once we substitute this expression for $H_2'$ in (\ref{sH}) and shift the contour of integration.
\end{proof}  
It now follows from (\ref{rewr-zeta}) and (\ref{rewr-L}) that
\begin{align*}
S_\zeta + S_L = \frac{1}{\log K} \sum_{n \geqslant 1} \frac{\Lambda (n) ( \chi(n) - 1 )}{n} \widehat{f} \lp \frac{\log n}{\log K} \rp.
\end{align*}
Since those $d$ we are interested in satisfy $d \not \equiv 1$ (mod $4$), we know that $\chi(n)$ is a quadratic Dirichlet character of modulus $4d = d^3$ (in case $d = 2$) or modulus $d$ (otherwise). This fact means that for all prime powers $n$ appearing in the above sum, $\chi(n) = 0$ if and only if $n$ is a power of $d$. Moreover, since $\chi(n) - 1 = 0$ whenever $n$ is the power of a split prime or an even power of an inert prime, the computation above shows that
\begin{align}\label{unification01}
S_\zeta + S_L &= - \frac{\log d}{\log K} \sum_{n \geqslant 1} d^{-n} \widehat{f} \lp \frac{n \log d}{\log K} \rp - \frac{2}{\log K} \sum_{n \geqslant 0} \hsp \sum_{\text{$p$ inert}} \frac{\log p}{p^{2n+1}} \widehat{f} \lp \frac{(2n+1) \log p}{\log K} \rp.
\end{align}
Regarding the first infinite sum, we note with the help of Lemma \ref{lemma4.2equiv} that
\begin{align}\label{unification02}
\begin{split}
- \frac{\log d}{\log K} \sum_{n \geqslant 1} d^{-n} \widehat{f} \lp \frac{n \log d}{\log K} \rp &= S_\mathrm{ram} + \frac{\log d}{\log K} \sum_{n \geqslant 0} d^{-n-1/2} \widehat{f} \lp \frac{(2n+1) \log d}{2 \log K} \rp \\
&= S_\mathrm{ram} - S_d.
\end{split}
\end{align}
Similarly, in the case of the second infinite sum, we see that
\begin{align}\label{unification03}
\begin{split}
&\hspace{-1em}-\frac{2}{\log K} \sum_{n \geqslant 0} \hsp \sum_{\text{$p$ inert}} \frac{\log p}{p^{2n+1}} \widehat{f} \lp \frac{(2n+1) \log p}{\log K} \rp \\
&\quad = S_\mathrm{inert} + \frac{2}{\log K} \sum_{\text{$p$ inert}} \hsp  \sum_{n \geqslant 1} \frac{\log p}{p^{2n}} \widehat{f} \lp \frac{2n \log p}{\log K} \rp \\
&\quad = S_\mathrm{inert} + \frac{2}{\log K} \sum_{p \geqslant 3 \atop \text{$p$ inert}} \sum_{n \geqslant 1} \frac{\log p}{p^{2n}} \widehat{f} \lp \frac{2n \log p}{\log K} \rp \\
&\quad \quad + \mathbbm{1} \lp \text{$2$ inert} \, \rp \frac{2 \log 2}{\log K} \, \sum_{n \geqslant 1} \, 2^{-2n} \widehat{f} \lp \frac{2n \log 2}{\log K} \rp \\[+0.5em]
&\quad = S_\mathrm{inert} - S_{A'} + \mathbbm{1} \lp 2 \text{ inert} \rp \frac{2 \log 2}{\log K} \sum_{n \geqslant 1} 2^{-2n} \widehat{f} \lp \frac{2n \log 2}{\log K} \rp,
\end{split}
\end{align}
again using Lemma \ref{lemma4.2equiv}. Since $2$ is inert in $\K$ if and only if $d \neq 2, \, 7$, we see that when $d \neq 2, \, 7$, the last term above is
\begin{align*}
\frac{2 \log 2}{\log K} \sum_{n \geqslant 1} 2^{-2n} \widehat{f} \lp \frac{2n \log 2}{\log K} \rp 
&= \frac{4}{2} \frac{\log 2}{\log K} \sum_{n \geqslant 0} 2^{-2n-2} \widehat{f} \lp \frac{(2n+2) \log 2}{\log K} \rp \\
&= \frac{1}{2} \frac{\log 2}{\log K} \sum_{n \geqslant 0} 2^{-2n} \widehat{f} \lp \frac{(2n+2) \log 2}{\log K} \rp \\
&= -S_H.
\end{align*}
Since $S_H$ furthermore vanishes if $d = 2, \, 7$, combining (\ref{unification01}), (\ref{unification02}), and (\ref{unification03}), we therefore obtain that
\begin{align}\label{unification04}
S_\mathrm{inert} + S_\mathrm{ram} = S_\zeta + S_L + S_{A'} + S_d + S_H .
\end{align}
It now follows from (\ref{unconditionaloneleveldensity}) and (\ref{oneleveldensity}) that our unconditional expression for the $1$-level density $D \lp \mathcal{F}(K); f \rp$ agrees with the expression conditional on the Ratios Conjecture, and hence with the Katz--Sarnak prediction (\ref{katzsarnak}), if
\begin{align*}
S_\mathrm{split} \approx S_J.
\end{align*}
\section{Comparison with the Katz--Sarnak Density Conjecture}
The goal of this section is to unify our explicit computation with the prediction of the Katz--Sarnak Density Conjecture. As we described earlier, this goal amounts to verifying that the term $S_\mathrm{split}$ coming from the split rational primes is equal to the term $S_J$ predicted by the Ratios Conjecture, at least up to some small error. \par 
We now generalize the result \cite[Lemma 2.1]{rudwax} which provides a useful relation between the angle $\theta_I$ and the norm $\norm (I)$ of a non-zero ideal $I \subset \ringint_\K$. 
\begin{theorem}\label{anglesandlengthsyeeeah}
Let $\Lambda = g \mathbb{Z}^2 \subset \R^2$ be a unimodular lattice with
\begin{align*}
g = \left( \begin{matrix}
a & 0 \\
0 & a^{-1}
\end{matrix} \right)
 \left( \begin{matrix}
1 & 0 \\
n & 1
\end{matrix} \right) \in \mathrm{SL}_2 (\mathbb{R}).
\end{align*}
Let $\ell$ be a line through the origin such that the angle between the positive $x$-axis and $\ell$ is $\theta \in \lb 0, \pi \rp$. If $\theta = \pi / 2$ (in which case we let $q := 1$), or if $-n + a^2 \tan \theta$ is an algebraic number of degree $q \geqslant 1$, then there exists $C = C(\Lambda, \ell) > 0$ such that for every $\textbf{v} \in \Lambda \setminus \ell$,
\begin{align*}
\left| \alpha (\textbf{v}) \right| \geqslant \frac{C}{\Vert \textbf{v} \Vert^{q}},
\end{align*}
where $\alpha (\textbf{v})$ denotes the angle between $\textbf{v}$ and $\ell$.
\end{theorem} 
\begin{proof}
Let us first note that if $\ell$ is the $y$-axis, then the claim follows easily: Indeed, the set of first coordinates of lattice points in $\Lambda$ is discrete, so for any $\textbf{v} \in \Lambda$, if the angle $\alpha(\textbf{v})$ between $\textbf{v}$ and the $y$-axis is non-zero, but small, we have
\begin{align*}
2 |\alpha (\textbf{v}) | \geqslant | \tan \alpha (\textbf{v}) | = \frac{|as|}{\left| (sn+t)/a \right|} \gg_a \frac{1}{\left| (sn+t)/a \right|} \geqslant \frac{1}{\sqrt{(sn+t)^2/a^2 + a^2 s^2}} = \frac{1}{\Vert \textbf{v} \Vert}.
\end{align*}
We can therefore assume that $\theta \neq \pi / 2$, so that $\cos \theta \neq 0$. \par Similarly,  if $\ell$ is the $x$-axis, Liouville's Theorem \cite[Thm. 1.1]{baker} and the assumption that $-n$ is algebraic of degree $q$ imply that $|sn+t| \gg |s|^{1-q}$ uniformly in $(s,t) \in \Z^2 \setminus \lbr {\bf 0} \rbr$, and hence
\begin{align*}
2 | \alpha (\textbf{v}) | &\geqslant | \tan \alpha (\textbf{v}) | = \frac{|sn+t|}{a^2 |s|} \gg_a \frac{1}{| a^q s^q|} \geqslant \frac{1}{\sqrt{a^{2q} s^{2q} + (sn+t)^{2q}/a^{2q}}} = \frac{1}{\Vert \textbf{v} \Vert^q},
\end{align*}
if $\alpha (\textbf{v})$ is non-zero and small. (Note that, in particular, this implies that $s \neq 0$ and $sn +t \neq 0$.) We can therefore also assume that $\theta \neq 0$, so that $\sin \theta \neq 0$. \par 
Let us now rotate $\ell$ and $\Lambda$ clockwise by the angle $\theta$, which transforms $\ell$ into the $x$-axis and $\Lambda$ into the lattice
\begin{align*}
\Lambda' = \lbr \lp 
\begin{matrix}
x \cdot \cos \theta + y \cdot \sin \theta \\
-x \cdot \sin \theta + y \cdot \cos \theta
\end{matrix} \rp : x = as, \, y = (sn + t)/a, \, s,t \in \Z \rbr.
\end{align*}
Since the substance of the claim pertains to the situation where $\alpha (\textbf{v}) \neq 0$ is very small, we assume that $\textbf{v} \in \Lambda'$ is any non-zero lattice point with $\alpha (\textbf{v})$ non-zero, but small. Under this assumption, we have the estimate
\begin{align}\label{pruttepuden}
\begin{split}
2 \left| \alpha (\textbf{v}) \right| \geqslant \left| \tan \alpha ( \textbf{v} ) \right| 
= \frac{| -x \cdot \sin \theta + y \cdot \cos \theta |}{| x \cdot \cos \theta + y \cdot \sin \theta |} 
= \frac{|y - x \cdot \tan \theta |}{|x + y \cdot \tan \theta|} 
\gg_a \frac{|t -  s \cdot \lp a^2 \tan \theta - n \rp|}{|x + y \cdot \tan \theta|}.
\end{split}
\end{align}
If $s = 0$, we have $x = 0$ and $y = t/a \neq 0$, and and the right-hand side of \eqref{pruttepuden} is bounded from below (independently of $\textbf{v}$) by $|a / \tan \theta | > 0$. We can therefore suppose that $s \neq 0$. In this case, the assumption about $a^2 \tan \theta - n$ implies, by Liouville's Theorem, that there exists $C = C(\Lambda, \ell) > 0$ such that
\begin{align}\label{bakerstuff}
\left| t - s \cdot (a^2 \tan \theta - n) \right| \geqslant \frac{C}{|s|^{q-1}}.
\end{align}
Since we have $|s|^{q-1} \ll_a \lp a^2 s^2 + (sn+t)^2/a^2 \rp^{(q-1)/2} = \Vert \textbf{v} \Vert^{q-1}$ and
\begin{align*}
|x + y \cdot \tan \theta | \leqslant \sqrt{ |x + y \cdot \tan \theta |^2 + |y - x \cdot \tan \theta |^2} = \frac{\Vert \textbf{v} \Vert}{|\cos \theta |},
\end{align*}
we obtain the claim from (\ref{pruttepuden}). 
\end{proof}
\textsc{Remark.} In anticipation of the lemma below, we use Theorem \ref{anglesandlengthsyeeeah} to define
\begin{align*}
Q := \max_{0 \leqslant m \leqslant 2N-1} \Big\{ q \geqslant 1 : {\small \begin{array}{l}
 -n(\ringint_\K) + a(\ringint_\K)^2 \tan \lp \pi m / (2N) \rp\\
\text{ is algebraic of degree $q$}
\end{array}} \Big\},
\end{align*}
where $n( \ringint_\K)$ and $a ( \ringint_\K)$ denote the parameters appearing in the Iwasawa decomposition of the lattice $\ringint_\K$ from Lemma \ref{iwasawaforringint}. That is, $Q$ is the largest of the degrees of all the algebraic numbers $-n(\ringint_\K) + a(\ringint_\K)^2 \tan \lp \pi m / (2N) \rp$, where $m = 0, \ldots, 2N-1$. Note that these numbers are indeed algebraic: This follows from Lemma \ref{iwasawaforringint} and from the fact that $\tan (\pi m / 2N)$ is algebraic, as $\tan (\pi m) = 0$ can be written as a quotient of two polynomials in $\tan (\pi m / 2N)$ with integer coefficients. \par 
In particular, there exists a constant $0 < c_0 < 1/4N$, which only depends on $\ringint_\K$ and $N$, such that for $m = 0, \ldots, 2N-1$, we have $| \alpha (\textbf{v}) | \geqslant c_0/ \Vert \textbf{v} \Vert^Q$, where $\alpha (\textbf{v})$ denotes the angle between $\ell_m$ and $\textbf{v} \in \ringint_\K \setminus \ell_m$. \\ \par 
We can now repeat the argument of Waxman in \cite[Section 6]{waxman} to prove that, at least when $\alpha := \sup \mathrm{supp} \, \widehat{f} < 1$, the unconditional asymptotic for the one-level density obtained above is in agreement with the prediction of the Katz--Sarnak Density Conjecture. Observe that when $\alpha < 1$, this is the case precisely if $S_\mathrm{split}$ is very small, cf. Lemma \ref{SJ-est}. \par 
\begin{lemma}\label{ssplitissmall} Suppose that $\alpha < 1$. Then $S_\mathrm{split} \ll_{\widehat{f}, \varepsilon} K^{\alpha - 1 + \varepsilon}$.
\end{lemma}
\begin{proof}
Note that the character sum appearing in the definition of $S_\mathrm{split}$
satisfies
\begin{align}\label{prelcharsumest}
\left| \sum_{k = 1}^K \psi_k (I) \right| = \left| \sum_{k = 1}^K e^{2iNk \theta_I} \right| \leqslant \frac{2}{\left| e^{2iN\theta_I} - 1 \right|}
\end{align}
whenever $\theta_I$ is not a multiple of $\pi / N$. Also, note that $\theta_I$ can't even be a multiple of $\pi / 2N$, for in that case, if $I = \langle \beta \rangle$, we have $( \beta / \overline{\beta} )^{N} = \pm 1$, and hence 
\begin{align*}
\langle \beta \rangle^N = \langle \beta^N \rangle = \langle \overline{\beta}^N \rangle = \langle \overline{\beta} \rangle^N,
\end{align*}
which forces $ \beta \equiv \overline{\beta}$ (mod $\ringint_\K^\times$) due to the unique factorization of ideals in $\ringint_\K$. It therefore follows that $\langle \beta \rangle$ can't lie over a split prime, which is a contradiction. Now, if 
\begin{align}\label{cuttingupthehalfcircle}
\frac{n \pi}{2N} < \theta_I < \frac{(n+1) \pi}{2N} 
\end{align}
for some $n = 0, \ldots, 2N-1$, we have
\begin{align*}
\frac{2}{\left| e^{2iN \theta_I} - 1 \right|} \ll \frac{1}{(n+1) \pi / (2N) - \theta_I} + \frac{1}{\theta_I - n \pi / (2N)}. 
\end{align*}
Since we can always take $\theta_I \in \lp 0, \pi \rp$ because $\pm 1 \in \ringint_\K^\times$ for all possible values of $d$, we can define $n \lp \theta_I \rp$ to be the unique $n \in \lbr 0, \ldots, 2N-1 \rbr$ such that (\ref{cuttingupthehalfcircle}) is satisfied. In combination with (\ref{prelcharsumest}), we therefore obtain
\begin{align}\label{charsumestimate}
\left| \sum_{k = 1}^K \psi_k (I) \right| \ll \frac{1}{(n(\theta_I)+1) \pi / (2N) - \theta_I} + \frac{1}{\theta_I - n(\theta_I) \pi / (2N)}.
\end{align}
\par Now, from the definition of $S_\mathrm{split}$ and (\ref{charsumestimate}) it is clear that 
\begin{align}\label{splitidealsum}
\begin{split}
S_\mathrm{split} &\ll \frac{1}{K \log K} \sum_{n = 0}^{2N-1} \sum_{I \subset \ringint_\K \atop n(\theta_I)=n} \hsp \frac{\Lambda \lp N(I) \rp}{\sqrt{N(I)}} \cdot \left| \widehat{f} \lp \frac{\log \sqrt{N(I)}}{\log K} \rp \right| \\
&\quad \times \lp \frac{1}{(n+1) \pi / (2N) - \theta_I} + \frac{1}{\theta_I - n \pi / (2N)} \rp
\end{split}
\end{align}
since all the powers $p^m$ ($m \geqslant 1$) of a split prime will appear as norms of suitable ideals $I \subset \ringint_\K$ with $\theta_I \neq 0$. \par 
Since $\ringint_\K$ is a two-dimensional lattice in $\C$, the basic idea is that the right-hand side of (\ref{splitidealsum}) can be estimated rather sharply by replacing it with an integral over certain parts of the ambient complex plane, where a change to polar coordinates will simplify the integrand greatly. To specify these domains of integration, for each $I \subset \ringint_\K$ we let
\begin{align*}
S_I \subset \lbr re^{i \theta} \in \C : n(\theta_I) \pi / 2N + c_0 < \theta < (n(\theta_I)+ 1) \pi / 2N - c_0 \rbr
\end{align*}
be a small annulus sector with distance $\gg 1$ to the origin and containing the generator $\sqrt{\norm (I)} e^{i \theta_I}$ of the ideal $I$ in the direction given by $\theta_I$. Furthermore, we choose these sectors to have the same area and so that $S_I \cap S_J = \varnothing$ if $I \neq J$.  Now, we find that
\begin{align*}
&\frac{1}{K \log K} \sum_{n = 0}^{2N-1} \sum_{I \subset \ringint_\K \atop n(\theta_I ) = n} \hsp \frac{\Lambda \lp N(I) \rp}{\sqrt{N(I)}} \left| \widehat{f} \lp \frac{\log \sqrt{N(I)}}{\log K} \rp \right| \lp \frac{1}{(n+1) \pi / (2N) - \theta_I} + \frac{1}{\theta_I - n \pi / (2N)} \rp \\
&\quad = \frac{1}{K \log K} \sum_{n = 0}^{2N-1} \sum_{I \subset \ringint_\K \atop n(\theta_I ) = n} \int_{S_I} \log r \left| \widehat{f} \lp \frac{\log r}{\log K} \rp \right| \lp \frac{1}{(n+1) \pi / (2N) - \theta_I} + \frac{1}{\theta_I - n \pi / (2N)} \rp \hsp \dd r \, \dd \theta \\[+0.3em]
&\quad \ll \frac{1}{K \log K} \sum_{n = 0}^{2N-1} \int_{r = 1}^\infty \log r \left| \widehat{f} \lp \frac{\log r}{\log K} \rp \right| \int_{n \pi / 2N+c_0/r^Q}^{(n+1) \pi / 2N - c_0/r^Q} \frac{1}{(n+1) \pi / 2N - \theta} \hsp \dd \theta \, \dd r \\[+1em]
&\quad \ll \frac{1}{K \log K} \int_{r = 1}^\infty \log r \left| \widehat{f} \lp \frac{\log r}{\log K} \rp \right| \int_{c_0/r^Q}^{\pi / 2N-c_0/r^Q} \frac{1}{\theta} \hsp \dd \theta \, \dd r,
\end{align*}
where we used Lemma \ref{anglesandlengthsyeeeah} and the remark following it to bound the angular parameter $\theta$ in terms of $r$. Continuing, we obtain from (\ref{splitidealsum}) and this estimate that
\begin{align*}
S_\mathrm{split} \ll_Q \frac{1}{K \log K} \int_{r = 1}^\infty \big( \log r \big)^2 \left| \widehat{f} \lp \frac{\log r}{\log K} \rp \right| \, \dd r \ll_{\widehat{f}} \frac{1}{K \log K} \int_{r = 0}^{K^\alpha} \big( \log r \big)^2 \, \dd r \ll K^{\alpha - 1} \log K.
\end{align*}
This concludes the proof.
\end{proof}
By the remark after (\ref{unification04}), we have therefore proved Theorem \ref{THMunconditional}.\\ \par 
It would be satisfactory to give a reason why $S_\mathrm{split}$ and $S_J$ would have anything to do with each other, so that the approximate equality given by Lemma \ref{ssplitissmall} is not just "accidental." It seems difficult to give any such reason, considering the particularly involved arithmetic quality of the term $S_J$, and we have not succeeded in this matter. However, we can make the following observation, which may or may not be relevant in order to make further progress: The character sum appearing in $S_\mathrm{split}$ is expressible in terms of the \textit{Dirichlet kernel} 
\begin{align*}
D_K (x) = \frac{\sin (Kx + x/2)}{\sin x/2} = 1 + 2 \, \sum_{k = 1}^K \cos (kx). 
\end{align*}
More precisely, we have the following result.
\begin{prop}\label{dirichletkernel}
Suppose $\mathfrak{q}_1$ and $\mathfrak{q}_2$ are different prime ideals in $\ringint_\K$ lying over a rational prime $p$, and let $n \geqslant 1$. Then we have
\begin{align*}
\sum_{k = 1}^K \Big( \psi_k(\mathfrak{q}_1^n) + \psi_k(\mathfrak{q}_2^n) \Big) = -1 + D_K \big( 2Nn \theta_{\mathfrak{q}_1} \big),
\end{align*}
where $\theta_{\mathfrak{q}_1} \in \lp 0, \pi \rp$ denotes the argument of a generator of $\mathfrak{q}_1$ in the upper half-plane.
\end{prop}
\begin{proof}
Suppose that $z = a + ib$ is any generator of $\mathfrak{q}_1$. Since $\bigcap_d \ringint_\K^\times = \Z / 2 \Z$, we can assume that $b > 0$. Moreover, since the generator of the ideal $\mathfrak{q}_1 \mathfrak{q}_2$ is a rational prime, the conjugate $\overline{z} = a - ib$ must be a generator of $\mathfrak{q}_2$. Therefore,
\begin{align*}
\psi_k(\mathfrak{q}_1^n) + \psi_k(\mathfrak{q}_2^n) &= \lp \frac{z}{\overline{z}} \rp^{Nnk} + \lp \frac{\overline{z}}{z} \rp^{Nnk} \\[+0.3em]
&= \exp \big( Nnk \log\big( z/\, \overline{z} \big) \big) + \exp \big(\! -\! \! Nnk \log\big( z/\, \overline{z} \big)\big) \\[+1em]
&= 2 \cos \big( Ni nk \log\big( z/\, \overline{z} \big)  \big) \\[+0.3em]
&= 2 \cos \lp Nink \log \frac{a/b + i}{a/b - i} \rp \\
&= 2 \cos \lp 2N n k \cdot \frac{1}{2i} \log \frac{a/b - i}{a/b + i} \rp \\[+0.5em]
&= 2 \cos \big( 2Nnk \cdot \arctan(b/a) \big), 
\end{align*}
where we used the identity
\begin{align*}
\arctan \frac{1}{x} = -\frac{1}{2i} \log \frac{x-i}{x+i},
\end{align*}
valid for all $x \neq 0$. Moreover, since the angle $\theta_{\mathfrak{q}_1}$ between $1$ and $a + ib$ is given by 
\begin{align*}
\theta_{\mathfrak{q}_1} = \begin{cases}
\arctan(b/a) &\text{if } a > 0, \\
\arctan(b/a) + \pi &\text{if } a < 0, 
\end{cases}
\end{align*}
the computation above and the periodicity of cosine show that
\begin{align*}
\psi_k(\mathfrak{q}_1^n) + \psi_k(\mathfrak{q}_2^n) = 2 \cos \big( 2Nnk\theta_{\mathfrak{q}_1} \big).
\end{align*}
It now follows that
\begin{align*}
\sum_{k = 1}^K \lp \psi_k(\mathfrak{q}_1^n) + \psi_k(\mathfrak{q}_2^n) \rp = 2 \sum_{k = 1}^K \cos \big( 2Nnk \theta_{\mathfrak{q}_1} \big) = -1 + D_K \big( 2Nn \theta_{\mathfrak{q}_1} \big),
\end{align*}
which completes the proof.
\end{proof}
Making use of this observation, we can then rewrite $S_\mathrm{split}$ as 
\begin{align*}
S_\mathrm{split} &= \frac{1}{K \log K} \sum_{(-d/p) = 1} \hsp \sum_{n \geqslant 1} \frac{\log p}{p^{n/2}} \widehat{f} \lp \frac{n \log p}{2 \log K} \rp \\ 
&\quad \quad - \frac{1}{\log K}  \sum_{(-d/p) = 1} \hsp \sum_{n \geqslant 1} \frac{\log p}{p^{n/2}} \widehat{f} \lp \frac{n \log p}{2 \log K} \rp \frac{D_K \big( 2Nn \theta_{p} \big)}{K},
\end{align*}
where $\theta_p \in (0, \pi )$ denotes the argument of one of the generators of either of the prime ideals lying over $p$. Thus, the primary benefit of Lemma \ref{dirichletkernel} is that it allows one to express the contribution of the split primes as a sum of two terms, among which the first one can be understood with partial summation, and the other one involves a well-studied function in a normalized form (due to the factor $1/K$). At any rate, the presence of the angles $\theta_p$ in the last contribution seems to be quite a difficult obstacle. A better understanding of these angles is likely a crucial step if one wishes to understand the relation $S_\mathrm{split} \approx S_J$ on a deeper level.

\end{document}